\documentclass[11pt, twoside]{article}
\pdfoutput=1

\usepackage{graphicx}
\usepackage{amsmath}
\allowdisplaybreaks

\usepackage{amssymb,amsfonts}
\usepackage{amsthm}
\usepackage{cite}
\usepackage{hyperref}
\usepackage{bm}
\usepackage{mathrsfs}
\usepackage{amssymb}

\usepackage[margin=1in]{geometry}

\usepackage{algorithm}
\usepackage{algorithmic}

\usepackage{graphicx}
\graphicspath{{figures/}}

\usepackage[nameinlink]{cleveref}

\newtheorem{theorem}{Theorem}[section]
\newtheorem{lemma}[theorem]{Lemma}
\newtheorem{proposition}[theorem]{Proposition}

\newtheorem{corollary}[theorem]{Corollary}

\theoremstyle{definition}
\newtheorem{definition}[theorem]{Definition}

\theoremstyle{remark}
\newtheorem{remark}[theorem]{Remark}

\newcommand{\iprod}[3]{\left(#1,#2\right)_{#3} }

\usepackage{fancyhdr}
\begin{document}
	
	\title{New proper orthogonal decomposition approximation theory for PDE solution data}

\author{Sarah K.~Locke%
	\thanks{Department of Mathematics and Statistics, Missouri University of Science and Technology, Rolla, MO (\mbox{sld77@mst.edu}, \mbox{singlerj@mst.edu}).}
	\and
	John~R.~Singler%
	\footnotemark[1]
}

\maketitle

\begin{abstract}
In our previous work \cite{Singler14}, we considered the proper orthogonal decomposition (POD) of time varying PDE solution data taking values in two different Hilbert spaces. We considered various POD projections of the data and obtained new results concerning POD projection errors and error bounds for POD reduced order models of PDEs. In this work, we improve on our earlier results concerning POD projections by extending to a more general framework that allows for non-orthogonal POD projections and seminorms.  We obtain new exact error formulas and convergence results for POD data approximation errors, and also prove new pointwise convergence results and error bounds for POD projections.  We consider both the discrete and continuous cases of POD.  We also apply our results to several example problems, and show how the new results improve on previous work.
\end{abstract}



\section{Introduction}

Proper orthogonal decomposition (POD) is a model order reduction technique for partial differential equations (PDEs) and other mathematical models. With this method, modes are computed from simulation or experimental data and a Galerkin projection is used with these modes to reduce the model.  Because POD reduced order models often have very low dimension, they can be used to efficiently simulate computationally demanding problems. Therefore, POD has been used in many fields of study including fluid dynamics and control theory. For a small selection of applications, see \cite{Nguyen17,Jin16,ROWLEY2005,Bergmann05,Willcox02}.  For more information about POD and many known results, see, e.g., \cite{GubischVolkwein17,HolmesLumleyBerkoozRowley12,Liang02}.

Because of the wide use of POD in many application areas, it is of great interest to study the approximation errors in POD model order reduction procedures.  Numerical analysis results for POD reduced order models of PDEs were first obtained by Kunisch and Volkwein \cite{KunischVolkwein01,KunischVolkwein02}, and then by many others; see, e.g., \cite{AllaFalconeVolkwein17,Chapelle12,ErogluKayaRebholz17,GiereIliescuJohnWells15,GraessleHinze18,GunzburgerIliescuSchneier20,GunzburgerJiangSchneier17,IliescuWang13,IliescuWang14_DQ,IliescuWang14,Jin16,KeanSchneier19,Kostova15,LiWuZhu19,MohebujjamanRebholzXieIliescu17,Rubino18,ShenSinglerZhang19,Singler14,Wang15,XieWellsWangIliescu18,ZerfasRebholzSchneierIliescu19,ZhuDedeQuarteroni17} and the references therein.


Understanding POD data approximation errors is typically important for these numerical analysis works.  To see this, let $w$ be the solution of the mathematical model, let $w_r$ be the solution of the POD reduced order model, and let $\pi_r$ be a projection onto the span of the first $ r $ POD modes.  Split the error as
$$
  w - w_r = \rho_r + \theta_r, \quad \rho_r = w - \pi_r w, \quad \theta_r = \pi_r w - w_r.
$$
Energy estimates can often be used to bound $\theta_r$ by quantities including various norms of $ \rho_r $, the POD data approximation error for that projection.

In our previous work \cite{Singler14}, exact error formulas and convergence results were proven for norms of $\rho_r$ involving two Hilbert spaces, where one space is a subset of the other. In that work, we considered the continuous POD setting and proved results for different combinations of POD spaces, projections, and norms.  Shortly after \cite{Singler14}, Iliescu and Wang \cite{IliescuWang14} provided analogous error formulas for the discrete POD case, and many of the recent numerical analysis works mentioned above use results from \cite{Singler14,IliescuWang14} or extensions of these results to other scenarios.

As POD is increasingly applied in a variety of situations, it becomes more useful to have error results that can be easily applied in a wide range of scenarios.  Therefore, in this work we extend POD data approximation results in \cite{Singler14,IliescuWang14} to a generalized framework that allows us to treat non-orthogonal POD projections and seminorms.  We prove new error formulas and convergence results for norms of quantities involving $\rho_r = w - \pi_r w$ with various POD projections $ \pi_r $.  We also prove new pointwise convergence results for different POD projections.  Non-orthogonal POD projections have been used in the numerical analysis for POD reduced order models \cite{IliescuWang13,Rubino18}; however, the exact POD data approximation error formulas and convergence results obtained here are new.  Exact POD data approximation errors using various seminorms have been obtained in some cases (see, e.g., \cite[Section 3.3]{GiereIliescuJohnWells15}, \cite[Lemma 3.1]{ShenSinglerZhang19}); the general extension and convergence results in this work are new.  Finally, some pointwise convergence results for POD projections were obtained in our earlier work \cite{Singler14}; we obtain new error bounds and improved convergence results here.

The POD data approximation error formulas presented in this work are exact and do not require the use of POD inverse inequalities. We consider both the discrete and continuous cases for POD and generalize the setting in \cite{Singler14,IliescuWang14} to allow a linear mapping between two Hilbert spaces to act on the data.  We require minimal assumptions on the data, the linear operator, and the Hilbert spaces; the assumptions we do require are naturally satisfied in many applications and also allow us to obtain convergence results even in the fully continuous case when the data has infinitely many positive POD eigenvalues.  We note that most of the proof strategies in this work are new; some proofs do rely on techniques from \cite{Singler14,Singler15}.


The rest of the paper is outlined as follows. \Cref{Section:basic} provides both a brief general background and POD specific background for both the discrete and continuous cases. Then \Cref{Section:mainassumptions} provides an overview of the new results along with the notation and main assumptions needed. Properties of POD and POD projections are given in \Cref{Section:basicresults}. Error formulas are presented in \Cref{Section:errors} and pointwise convergence results are given in \Cref{Section:convergence}. Finally, in \Cref{Section:Examples}, we consider examples and compare the results from previous work and the current work. 

\section{Background}
\label{Section:basic}

In this section, we recall some functional analysis background material, and also the basic theory for discrete POD and continuous POD.  For details and proofs for the basic discrete and continuous POD theory, see, e.g., \cite{Djouadi08,GubischVolkwein17,HolmesLumleyBerkoozRowley12,KunischVolkwein02,Volkwein04,QuarteroniManzoniNegri16} and also \Cref{sec:optimality_POD_proof}.  


\subsection{Functional Analysis Background}
\label{sec:FA_background}

Let $V$ and $W$ be Hilbert spaces with inner products\footnote{In this paper, all inner products and sesquilinear forms are linear in the first argument and conjugate linear in the second argument.} $ (\cdot,\cdot)_V $ and $ (\cdot,\cdot)_W $ and corresponding norms $ \| \cdot \|_V $ and $ \| \cdot \|_W $.  Throughout this work, the scalar field $ \mathbb{K}$ for all spaces is either $ \mathbb{K} = \mathbb{R} $ or $ \mathbb{K} = \mathbb{C} $.

\textbf{Linear Operators}:  Let $T:V \to W$ be a linear operator with domain $ \mathcal{D}(T) \subset V$, range $ \mathcal{R}(T) \subset W $, and null space $ \mathrm{ker}(T) \subset V $.  The \textit{rank} of $ T $ is the dimension of $ \mathcal{R}(T) $.
The operator $T$ is \textit{bounded} if $ \|Tv\|_W \leq M \|v\|_V$ for all $v \in \mathcal{D}(T)$. Throughout this paper, we only consider bounded operators $ T: V \to W $ that are defined on the whole space, so $\mathcal{D}(T) = V$.  For such a bounded operator $ T : V \to W $, the usual operator norm is given by $ \| T \| = \sup\{ \| T v \|_W : v \in V, \| v \|_V = 1 \} $.  We also consider unbounded linear operators that are not defined everywhere, so that $\mathcal{D}(T) \neq V$.  The operator $ T $ is \textit{closed} if its graph, $\mathcal{G}(T) = \{(v,w) : v \in \mathcal{D}(T), w = Tv  \}$, is closed in $V\times W$.  If $ T $ is bounded (and everywhere defined), then $ T $ is closed.  If $T$ is closed and invertible, then $T^{-1}$ is closed.
%

\textbf{Adjoint Operators:} The \textit{Hilbert-adjoint operator} $ T^* : W \to V $ satisfies $ (T v, w )_W = (v, T^* w)_V $ for all $ v \in \mathcal{D}(T) $ and $ w \in \mathcal{D}(T^*) $.  If $ T $ is bounded, then $ T^* $ exists, is unique, and is also bounded.  If $ T $ is densely defined, then $ T^* $ exists, is unique, and is closed; in addition, if $ T $ is closed, then $ T^* $ is densely defined.  If $ T : V \to W $ is invertible, then we let $ T^{-*} : V \to W $ denote the Hilbert adjoint operator of the inverse $ T^{-1} : W \to V $.  We note for $T^*$ to exist we need $T$ bounded or densely defined, and for $T^{-*}$ to exist we need $T^{-1}$ bounded or densely defined.  We note these assumptions when necessary.

The following basic result is important in this work.
\begin{lemma} \label{lem1}
	Let $V$ and $ W $ be Hilbert spaces.  If $T:V\to W$ is a bounded linear operator, then $\mathrm{ker}(T T^*) = \mathrm{ker}(T^*)$ and $\mathrm{ker}(T^* T) = \mathrm{ker}(T)$.
\end{lemma}
\begin{proof}	
	We only prove the first one. Let $w \in \mathrm{ker}(T T^*)$. Then,
	$$TT^* w = 0 \Rightarrow (TT^* w,w)_W = 0 \Rightarrow (T^*w,T^*w)_{V}=0 \Rightarrow \|T^*w\|^2_{V} = 0 \Rightarrow T^*w = 0.$$
	Next, let $w \in \mathrm{ker}(T^*)$. Then
	$T^*w=0 \Rightarrow TT^*w =0$. 
\end{proof}


\textbf{Projections}:  A bounded linear operator $\Pi: V \to V$ is a \textit{projection} onto $U = \mathcal{R}(\Pi) $ if $\Pi^2 = \Pi$.  Then we have $\Pi v \in U$ for all $ v \in V $ and $\Pi u = u$ for all $ u \in U $.  Also, $\Pi $ is an \textit{orthogonal projection} if $ u = \Pi v \in U$ minimizes $\inf_{u \in U}\|v-u\|_V$ for any $v \in V$.  A nontrivial orthogonal projection $ \Pi $ is automatically self-adjoint, i.e., $ \Pi^* = \Pi $, and satisfies $ \| \Pi \| = 1 $.  We consider non-orthogonal projections in this work, and therefore we do not assume a projection is orthogonal or self-adjoint unless explicitly specified.  Sometimes, we assume a family of projections $ \{ \Pi_r \} $ is uniformly bounded in operator norm, i.e., there exists a constant $ C $ such that $ \| \Pi_r \| \leq C $ for all $ r $.


\textbf{The Singular Value Decomposition of a Compact Operator:} If $T: V \to W$ is a compact linear operator, with separable Hilbert spaces $V$ and $W$, then $T$ has a \textit{singular value decomposition} (SVD). The positive singular values of $T$ are defined to be the square roots of the positive eigenvalues of the self-adjoint nonnegative compact operators $TT^* : W \to W $ and $T^*T : V \to V$.  Further, the nonzero eigenvalues of these operators are equal, and we consider zero a singular value of $ T $ if either operator has a zero eigenvalue. If the ordered singular values of $T$ are given by $ \mu_1 \geq \mu_2 \geq \cdots \geq 0 $ (including repetitions), the orthonormal basis of eigenvectors of $TT^*$ is given by $\{ \psi_k \} \subset W$, and the orthonormal basis of eigenvectors of $T^*T$ is given by $\{ g_k \} \subset V$, then the singular value decomposition of $ T $ is the expansion given by
\begin{equation*}
Tg = \sum_{k\geq 1} \mu_k (g, g_k)_V \psi_k
\end{equation*}
for all $g \in V$. If $\mu_k > 0$, then 
$$
Tg_k = \mu_k \psi_k \text{ and } T^*\psi_k = \mu_k g_k.
$$
Also, the rank $ r $ truncated SVD $ T_r : V \to W $ of $ T $ is defined for $ g \in V $ by
$$
T_r g := \sum_{k = 1}^r \mu_k (g, g_k)_V \psi_k.
$$
For more information, see, e.g., \cite[Chapters VI--VIII]{GohbergGoldbergKaashoek90}, \cite[Section V.2.3]{Kato95}, \cite[Chapter 30]{Lax02}, \cite[Sections VI.5--VI.6]{ReedSimon80}.
%

\textbf{Hilbert-Schmidt Operators:} Let $T: V \to W$ be a linear operator, with separable Hilbert spaces $V$ and $W$, and let $\{g_k\}$ be any orthonormal basis for $V$. Define the \textit{Hilbert-Schmidt norm} of $ T $ as
\begin{equation}
\|T\|_{\mathrm{HS}(V,W)} = \bigg(  \sum_{k \geq 1} \|T g_k\|_V^2 \bigg)^{1/2}.
\end{equation}
If the sum converges we say the operator $T$ is \textit{Hilbert-Schmidt}. The Hilbert-Schmidt norm is independent of choice of orthonormal basis, every Hilbert-Schmidt operator is compact, $ \| T \| \leq \| T \|_{\mathrm{HS}(V,W)} $, $T$ is Hilbert-Schmidt if and only if $T^*$ is Hilbert-Schmidt, and $ T $ is Hilbert-Schmidt if and only if $ \sum_{k \geq 1} \sigma_k^2 < \infty $, where $ \{ \sigma_k \} $ are the singular values (including repetitions) of $ T $.  We also have
$$
\|T\|^2_{\mathrm{HS}(V,W)} = \|T^*\|^2_{\mathrm{HS}(W,V)} = \sum_{k \geq 1} \sigma_k^2.
$$
For more, see, e.g., \cite[Chapter VIII]{GohbergGoldbergKaashoek90}, \cite[Section V.2.4]{Kato95}, \cite[Section VI.6]{ReedSimon80}.

\textbf{Bochner Spaces:}  Let $ \mathcal{O} $ be an open subset of $ \mathbb{R}^d $, for some $ d \geq 1 $.  For $ p \in [1,\infty) $, let $ L^p(\mathcal{O};V) $ denote the Bochner space of (equivalence classes of) Lebesgue measurable functions $ v : \mathcal{O} \to V $ satisfying $ \int_\mathcal{O} \| v(t) \|_V^p \, dt < \infty $.  For $ p = 2 $, $ L^2(\mathcal{O};V) $ is a Hilbert space with inner product
$$
(v,w)_{L^2(\mathcal{O};V)} = \int_\mathcal{O} ( v(t), w(t) )_V \, dt.
$$
The following theorem, see, e.g., \cite[Theorem III.6.20]{DunfordSchwartz58} and \cite[Theorem 4.2.10]{Miklavcic98}, allows us to bring a closed linear operator inside an integral.
\begin{theorem}\label{bochnerint}
	Suppose $T:\mathcal{D}(T) \subset V \to W$ is a closed linear operator. If $v: \mathcal{O} \to \mathcal{D}(T)$, $v \in L^1(\mathcal{O};V)$, and $Tv \in L^1(\mathcal{O};W)$, then 
	$$\int_\mathcal{O} v(t) \, dt \in \mathcal{D}(T) \quad \text{ and } \quad  T\int_\mathcal{O} v(t) \, dt = \int_\mathcal{O} Tv(t) \, dt.$$
\end{theorem}

\subsection{Discrete POD}\label{discretebackground}


Let $ X $ be a separable Hilbert space.  For the discrete case, let $s$ be a positive integer and assume the POD data is given by $\{w_j\}_{j=1}^s \subset X$. Let $ \mathbb{K} = \mathbb{R} $ or $ \mathbb{K} = \mathbb{C} $, and define $S := \mathbb{K}_\Gamma^s$ with the weighted inner product given by
$$(u,v)_S = v^*\Gamma u = \sum_{j=1}^{s} \gamma_j u_j \overline{v_j},$$
where $u,v \in S$, $\Gamma = \text{diag}(\gamma_1, \gamma_2, ..., \gamma_s)$, and the values $\{\gamma_j\}_{j=1}^s$ are positive weights. Note these weights commonly arise from integral approximations.
 Define the POD operator $ K : S \to X $ by
\begin{equation}\label{XPODdiscrete}
K f = \sum_{j=1}^s \gamma_j f_j \, w_j,  \quad  f = [ f_1, f_2, \ldots, f_s ]^T.
\end{equation}

Since $ K $ has finite dimensional range, it is a compact operator and has a singular value decomposition.  Let $ \{ \sigma_k, f_k, \varphi_k \} \subset \mathbb{R} \times S \times X $ be the singular values and orthonormal singular vectors ordered so that $ \sigma_1 \geq \sigma_2 \geq \cdots \geq 0 $.  Thus, the singular value decomposition is given by 
\begin{equation}\label{eqn:POD_operator_SVD}
K f = \sum_{j \geq 1} \sigma_{j} (f,f_j)_S \varphi_j.
\end{equation}
When $ \sigma_k > 0 $, we have
$$
K f_k = \sigma_k \varphi_k,  \quad \text{and}  \quad K^* \varphi_k = \sigma_k f_k,
$$
where $ K^* : X \to S $ is the Hilbert adjoint operator given by
\begin{equation*}
K^* x = [ (x,w_1)_X, (x,w_2)_X, \ldots, (x,w_s)_X ]^T.
\end{equation*}

For a positive integer $ r $, define $ X_r = \mathrm{span}\{ \varphi_k \}_{k=1}^r $.  Let $ \Pi_r^X : X \to X $ be the orthogonal projection onto $ X_r $, i.e., for $ x \in X $ fixed, $ \Pi^X_r x \in X_r $ minimizes the approximation error $ \| x - x_r \|_X $ over all choices of $ x_r \in X_r $.  Since $ \{ \varphi_k \} $ is an orthonormal set in $ X $, we have the exact representation
\begin{equation}\label{Pi^Xdef}
\Pi_r^X x = \sum_{k=1}^r (x,\varphi_k)_X \varphi_k.
\end{equation}

The singular vectors $ \{ \varphi_k \} $ are called the POD modes of the data $ \{ w_k \} \subset X $.  The POD modes provide the best low rank approximation to the data in the following sense: we have
\begin{equation}\label{knownPODerror}
\sum_{k=1}^s \gamma_k \| w_k - \Pi_r^X w_k \|_X^2  =  \sum_{k > r}  \sigma_k^2,
\end{equation}
and no other choice of an orthonormal basis in \eqref{Pi^Xdef} gives a smaller value for the approximation error.

\begin{definition}\label{def:POD_svalues_evalues_modes_s_X}
	We call the singular values $ \{ \sigma_k \} $ and singular vectors $ \{ \varphi_k \} \subset X $ of $ K $ the \textit{POD singular values} and \textit{POD modes} for the data $ \{ w_j \}_{j=1}^s $, respectively.  We also call the eigenvalues $ \{ \lambda_k \} $ of the operator $ K K^* : X \to X $ the \textit{POD eigenvalues} for the data $ \{ w_j \}_{j=1}^s $.  We let $ s_X $ denote the number of positive POD singular values (or positive POD eigenvalues) for the data $ \{ w_j \}_{j=1}^s $, i.e., $ s_X = \mathrm{rank}(K) $.
\end{definition}
From \Cref{sec:FA_background}, we know $ \lambda_k = \sigma_k^2 $ whenever $ \lambda_k > 0 $.  Also, we have $ s_X \leq s < \infty $.  It is possible for data to have a zero POD singular value, but have all positive POD eigenvalues; this can happen if $ s > \dim(X) $.



%

\subsection{Continuous POD}\label{continuousbackground}


Similarly to the discrete case we define the POD operator $K:S\to X$ for the continuous case, where again $ X $ is a separable Hilbert space. Let $d $ and $m$ be positive integers and let $\mathcal{O} \subset \mathbb{R}^d$ be an open set. Then define $S := L^2(\mathcal{O}; \mathbb{K}^m)$, where $ \mathbb{K} = \mathbb{R} $ or $ \mathbb{K} = \mathbb{C} $.  We note that $ L^2(\mathcal{O}) $ is separable (see, e.g., \cite[Theorem 2.5-4]{Ciarlet13}), and therefore so is $ S $.  Assume the POD data is given by $\{w_j\}_{j=1}^m \subset L^2(\mathcal{O};X)$. 
%

\begin{remark}
	In POD applications the set $ \mathcal{O}$ is frequently a time interval; however, researchers also take $ \mathcal{O}$ to be a multidimensional parameter domain as well.  Note that we could also consider multiple open sets, $\mathcal{O}_j \subset \mathbb{R}^{d_j}$, and data $ w_j \in L^2(\mathcal{O}_j;X) $ for $ j = 1, \ldots, m $.  In this case, we would define $S := L^2(\mathcal{O}_1) \times \cdots \times L^2(\mathcal{O}_m)$. All results in this paper hold for this case as well. The previous case is chosen to simplify notation.
\end{remark}

%

Define the POD operator $K: S \to X$ by 
\begin{equation}\label{XPODcontinuous}
K f = \sum_{j=1}^{m} \int_{\mathcal{O}} f_j(t) w_j(t) dt, \quad f \in S.
\end{equation}
Since $f \in S$, note that $f = [f_1,f_2,\ldots,f_m]^T$, where each $f_j \in L^2(\mathcal{O})$. As in the discrete case, we know that $K$ is a compact operator and has a singular value decomposition.  We let $ \{ \sigma_k, f_k, \varphi_k \} \subset \mathbb{R} \times S \times X $ denote the singular values and orthonormal singular vectors ordered so that $ \sigma_1 \geq \sigma_2 \geq \cdots \geq 0 $. The SVD of $ K $ is given as in the discrete case \eqref{eqn:POD_operator_SVD}. Thus, when $ \sigma_k > 0 $, we have
$$
K f_k = \sigma_k \varphi_k,  \quad\text{and}  \quad  K^* \varphi_k = \sigma_k f_k,
$$
where $ K^* : X \to S $ is the Hilbert adjoint operator defined by
\begin{equation*}
[K^* x](t) = [ (x,w_1(t))_X, (x,w_2(t))_X, \ldots, (x,w_m(t))_X ]^T.
\end{equation*}

We define $ X_r := \mathrm{span}\{ \varphi_k \}_{k=1}^r $ and the orthogonal projection $\Pi_r^X : X \to X$ \eqref{Pi^Xdef} as before.  The data approximation error is given by
\begin{equation}\label{knownPODerror_cont}
\sum_{j=1}^m \int_\mathcal{O} \| w_j(t) - \Pi_r^X w_j(t) \|_X^2 \, dt  =  \sum_{k > r}  \sigma_k^2,
\end{equation}
and the error goes to zero as $ r \to \infty $.  As in the discrete case, no other orthonormal basis in \eqref{Pi^Xdef} gives a smaller value for the error.

We define the POD singular values, POD modes, POD eigenvalues, and $ s_X = \mathrm{rank}(K) $ as in \Cref{def:POD_svalues_evalues_modes_s_X} for the discrete case.  Again, it is possible for data to have a zero POD singular value, but have all positive POD eigenvalues; an example where $ X $ is infinite dimensional can be found in \cite[Section 3.1, Example 3]{Singler15}.  Also, if $ X $ is finite dimensional, then the data always has a zero POD singular value.

\section{Main Assumptions, Notation, and New Results}\label{Section:mainassumptions}

In this section we highlight the notation used in each case as well as the main assumptions made throughout the paper. Further we briefly present an overview of the new results and give an example to illustrate how the new results can be used.

Throughout the remainder of this paper, assume $ X $ and $ Y $ are separable Hilbert spaces, and $L: \mathcal{D}(L) \subset X \to Y$ is a linear operator.  We study POD error formulas and POD projections involving the data $ \{ w_j \} $ and the data $ \{ L w_j \} $.

\subsection{Discrete Case: Assumptions and Notation}\label{discretebackground2}

Recall from \Cref{discretebackground} we consider data $ \{ w_j \}_{j=1}^s \subset X $ and the corresponding POD operator $ K : S \to X $ defined by $ K f = \sum_{j=1}^s \gamma_j f_j w_j $, where $ S = \mathbb{K}_\Gamma^s $ and $ \mathbb{K} $ is either $ \mathbb{R}$ or $ \mathbb{C}$.  The singular value decomposition of $ K $ is given by $ K f = \sum_{k \geq 1} \sigma_k ( f, f_k )_S \varphi_k $.  The set $ X_r $ is the span of $ \{ \varphi_k \}_{k=1}^r $, and $ \Pi_r^X : X \to X $ is the orthogonal projection onto $ X_r $.

To consider POD projections involving the data $ \{ L w_j \} $, we make the following assumption:
\begin{quote}
	\textbf{Main assumption:}  For the discrete case, we assume throughout the paper that (i) $ \{ w_j \}_{j=1}^s \subset \mathcal{D}(L) $, and also (ii) $ \sigma_r > 0 $ whenever we consider the projection $ \Pi_r^X $.
\end{quote}
Assumption (i) has two important consequences.  First, since $ w_j \in \mathcal{D}(L) $ for each $ j $, we know the range of $ K $ is contained in $ \mathcal{D}(L) $.  Second, assumption (i) allows us to consider the POD operator $ K^Y : S \to Y $ for the data $ \{ L w_j \}_{j=1}^s \subset Y $ defined by
\begin{equation}\label{YPODdiscete}
K^Y f = LK f =  \sum_{j = 1}^s \gamma_j f_j Lw_j,  \quad  f = [ f_1, f_2, \ldots, f_s ]^T.
\end{equation}
Note that $ K^Y $ is the result of applying $L$ to the POD operator $ K $ for the data $ \{ w_j \} $, i.e., $ K^Y = LK $.  Since $ K^Y $ has finite rank, it is compact and has a singular value decomposition.  Define $ s_Y = \mathrm{rank}(K^Y) $ to be the number of positive singular values of $ K^Y $.  Note that assumption (i) is automatically satisfied if $ L $ is bounded.

For assumption (ii), note that if $ \sigma_k > 0 $, then assumption (i) implies the corresponding singular vector $ \varphi_k $ is in $ \mathcal{D}(L) $ since
\begin{equation}\label{eqn:PODmodes_domL}
\varphi_k = \sigma_k^{-1} K f_k \in \mathcal{D}(L).
\end{equation}
Since $ \sigma_r > 0 $, this implies $ X_r \subset \mathcal{D}(L) $ and $ \Pi_r^X $ maps into $ \mathcal{D}(L) $.

\medskip

To guarantee the boundedness of certain POD projections, in some cases of \Cref{boundedext} we need to assume the POD modes $ \{ \varphi_k \}_{k=1}^r \subset \mathcal{D}(L) $ satisfy some additional regularity properties.  These properties can be guaranteed by making additional regularity assumptions on the data.

First, the condition $\{\varphi_k\}_{k=1}^r \subset \mathcal{D}(L^{-*})$ is guaranteed to hold if we assume $ \sigma_r > 0 $ and $w_j \in \mathcal{D}(L^{-*})$ for each $j$.  With this assumption, we know as above that $\mathcal{R}(K) \subset \mathcal{D}(L^{-*})$ and also $ \varphi_k \in \mathcal{D}(L^{-*}) $ whenever $ \sigma_k > 0 $.  Since $\sigma_r >0$, we can guarantee $\{\varphi_k\}_{k=1}^r \subset \mathcal{D}(L^{-*})$.

Next, a similar argument using \eqref{eqn:PODmodes_domL} shows the condition $\{L\varphi_k\}_{k=1}^r \subset \mathcal{D}(L^{*})$ is guaranteed to hold if we assume $ \sigma_r > 0 $ and $Lw_j \in \mathcal{D}(L^{*})$ for each $j$.

\subsection{Continuous Case: Assumptions and Notation}\label{contcase:mainassumption}

The continuous case requires a few more assumptions.  Recall $ K : S \to X $, where $S := L^2(\mathcal{O}; \mathbb{K}^m)$ and $ \mathbb{K} $ is either $ \mathbb{R}$ or $ \mathbb{C}$.  In order to define the POD operator $K^Y$ and ensure $ \{ \varphi_k \}_{k=1}^r \subset \mathcal{D}(L) $, we make the following assumption:
\begin{quote}
	\textbf{Main assumption:}  For the continuous case, we assume throughout the paper that (i) $ \{ Lw_j \}_{j=1}^m \subset L^2(\mathcal{O};Y) $, and for all $f \in S$ we have $Kf \in \mathcal{D}(L)$ and
	\begin{equation*}
	LKf = \sum_{j=1}^{m} \int_{\mathcal{O}} f_j(t) Lw_j(t)dt,
	\end{equation*}
	and also (ii) $ \sigma_r > 0 $ whenever we consider the projection $ \Pi_r^X $.
\end{quote}
%
%
As in the discrete case, assumption (i) gives $ \mathcal{R}(K) \subset \mathcal{D}(L) $ and allows us to define the (compact) POD operator $ K^Y = LK $ for the data $ \{ Lw_j \}_{j=1}^m \subset L^2(\mathcal{O};Y) $.  As before, we let $ s_Y = \mathrm{rank}(K^Y) $ be the number of positive singular values of $ K^Y $.  Also as in the discrete case, assumptions (i) and (ii) imply $ \{ \varphi_k \}_{k=1}^r \subset \mathcal{D}(L) $ and $ \Pi_r^X $ maps into $ \mathcal{D}(L) $.

\begin{remark}\label{remark:cont_assumption}
	There are three common conditions that guarantee assumption (i) holds. 
	\begin{enumerate}
		\item If $L:X\to Y$ is bounded, the operator $L$ can be pulled through the integral in the definition of $K$ and assumption (i) clearly holds. 
		
		\item\label{remark:cont_assumption_finite_sum} If each $w_j \in L^2(\mathcal{O};X)$ takes the form $$w_j(t) = \sum_{\ell,k = 1}^{n_j} a_{jk\ell} g_{kj} (t) x_{\ell j},$$
		where $a_{jk\ell}$ are constants in $\mathbb{K}$, $g_{kj} (t) \in L^2(\mathcal{O})$, and $ x_{\ell j} \in \mathcal{D}(L)$, then it can be checked that assumption (i) holds. 
		This condition is similar to the assumption made in the discrete case. 
		
		\item\label{remark:cont_assumption_closed} If $L:\mathcal{D}(L)\subset X \to Y$ is closed, $w_j \in \mathcal{D}(L)$ a.e., and $Lw_j \in L^2(\mathcal{O};Y)$ then \Cref{bochnerint} implies assumption (i) holds. 
	\end{enumerate}
\end{remark}



Again, for certain cases of \Cref{boundedext} we need to assume the POD modes $ \{ \varphi_k \}_{k=1}^r \subset \mathcal{D}(L) $ satisfy some additional regularity properties.  As in the discrete case, we can make additional assumptions on the data to satisfy these regularity properties.

We briefly mention conditions on the data similar to \Cref{remark:cont_assumption}, \Cref{remark:cont_assumption_closed} that yield the needed regularity.  First, if $L^{-*}$ exists, it is closed.  Therefore, $\{\varphi_k\}_{k=1}^r \subset \mathcal{D}(L^{-*})$ holds if we assume $ \sigma_r > 0 $, $w_j \in \mathcal{D}(L^{-*})$ a.e., and $ \{ L^{-*}w_j \}_{j=1}^m \in L^2(\mathcal{O};Y) $.  Second, if $ L^* $ exists, then it is closed.  Therefore, $\{L \varphi_k\}_{k=1}^r \subset \mathcal{D}(L^{*})$ holds if we assume $ \sigma_r > 0 $, $Lw_j \in \mathcal{D}(L^{*})$ a.e., and $ \{ L^{*}Lw_j \}_{j=1}^m \in L^2(\mathcal{O};X) $.

We also note that the condition in \Cref{remark:cont_assumption}, \Cref{remark:cont_assumption_finite_sum} can be modified similarly to the discrete case to yield the required regularity.



\subsection{An Overview of the New Results}
\label{sec:overview}

Here we give an overview of the new results presented in this paper. For the overview we focus on the continuous case, but there are analogous results for the discrete case.

Recall the standard POD orthogonal projection, $\Pi_r^X: X \to X$ given by \eqref{Pi^Xdef}, and the known POD data approximation error given by
$$
\sum_{j=1}^m \int_{\mathcal{O}} \| w_j(t) - \Pi_r^X w_j(t) \|_X^2 dt = \sum_{k > r} \sigma_{k}^2.
$$
One of the goals of this paper is to find extensions of this error formula to other scenarios involving the linear operator $ L : X \to Y $ and another sequence of projections, which need not be orthogonal.
\begin{definition}
	For a positive integer $r$ with $ \sigma_r > 0 $, we define $Y_r := LX_r = \mathrm{span}\{L\varphi_k\}_{k=1}^r$ and we let $\Pi_r^Y: Y \to Y$ be a projection onto $Y_r$. 
\end{definition}
\begin{remark}
	First, the condition $ \sigma_r > 0 $ implies $ X_r \subset \mathcal{D}(L) $ and so the definition makes sense.  We assume throughout that $ \sigma_r > 0 $ whenever we consider $ \Pi_r^Y $.  Next, it is important to note that unless stated otherwise we do not assume the projection $ \Pi_r^Y $ is orthogonal.  To obtain convergence results as $ r $ increases, we sometimes need to require $ \{ \Pi_r^Y \} $ are uniformly bounded in operator norm.  If $ \{ \Pi_r^Y \} $ are the orthogonal projections onto $ Y_r $, then this condition is satisfied.
\end{remark}



Under the main assumption we have the data approximation errors
\begin{equation}\label{error1}
\sum_{j=1}^m \int_{\mathcal{O}} \|L w_j(t) - L \Pi _r^X w_j(t) \|^2_{Y} dt  = \sum_{k>r} \sigma_k^2 \|L \varphi_k \|^2_Y
\end{equation}
and	
\begin{equation}\label{error2}
\sum_{j=1}^m \int_{\mathcal{O}}  \|L w_j(t) -  \Pi _r^Y L w_j(t) \|^2_{Y} dt  = \sum_{k>r} \sigma_k^2 \|L \varphi_k - \Pi _r^Y L \varphi_k \|^2_Y.
\end{equation}
The error in \eqref{error1} converges to zero as $r \to \infty$, and the error in \eqref{error2} tends to zero as $ r $ increases when the projections $ \{ \Pi_r^Y \}$ are uniformly bounded.  Also, under a basic condition on $L^{-1}$, we have the data approximation error
\begin{equation}\label{error3}
\sum_{j=1}^m \int_{\mathcal{O}}  \| w_j(t) - L^{-1} \Pi _r^Y L w_j(t) \|^2_{X} dt  = \sum_{k>r} \sigma_k^2 \| \varphi_k - L^{-1} \Pi _r^Y L \varphi_k \|^2_X .
\end{equation}
There are further conditions implying the error in \eqref{error3} converges to zero as well. The details for the assumptions, theorem statements, and proofs can be found in \Cref{Section:ContError} for the continuous case and \Cref{discreteerror} for the discrete case.

We also have pointwise convergence results in \Cref{Section:convergence} for both the discrete and continuous cases. One new result gives that if all POD eigenvalues for the data $\{Lw_j\}$ are nonzero and $ \{ \Pi_r^Y \}$ is uniformly bounded, then $\Pi_r^Y y \to y$ for all $y \in Y$ as $r$ increases. We also prove error bounds for pointwise convergence of the other projections considered. Boundedness of either $L$ or $L^{-1}$, along with various range conditions, also play important roles in the pointwise convergence of these POD projections and their mappings.

In the pointwise convergence result for $ \Pi_r^Y $ mentioned above, we required all of the POD eigenvalues for $\{Lw_j\}$ to be nonzero.  This improves on a similar result from our earlier work \cite{Singler14}, where we assumed all of the POD \textit{singular values} are nonzero.  The current result is less restrictive; see \Cref{discretebackground,continuousbackground}.  We also explore the boundedness of certain non-orthogonal POD projections in \Cref{sec:nonorth_POD_projections} and the relationship between the two sets of POD singular values for the data $ \{ w_j \} $ and the data $ \{ L w_j \} $ in \Cref{sec:POD_singular_values}.


\subsection{A Brief Example}

Next, we briefly present numerical results for an example to demonstrate our new results.  POD model order reduction is considered for this example in \cite{Wang15}; here, we focus on the POD data approximation errors.  The new results are discussed in greater detail for other examples in \Cref{Section:Examples}.

Consider a nerve impulse model, the FitzHugh-Nagumo system in one dimension. This model is given by 
\begin{align*}
\dfrac{\partial u(t,x)}{\partial t} &= \mu \dfrac{\partial^2 u(t,x)}{\partial x^2} - \dfrac{1}{\mu} v(t,x) + \dfrac{1}{\mu} f(u) + \dfrac{c}{\mu}, \quad 0<x<1, \\
\dfrac{\partial v(t,x)}{\partial t} &= b u(t,x) - \gamma v(t,x) + c, \quad 0<x<1,
\end{align*} 
where 
$$ f(u) = u(u-0.1)(1-u), $$
$\mu = 0.015$, $b = 0.5$, $\gamma = 2$, and $c = 0.05$. Further, the boundary conditions are given by 
$$ u_x(t,0) = -50000 t^3 e^{-15t} \quad \text{and} \quad u_x(t,1) = 0, $$
and the initial conditions are zero.

For this example, we take the Hilbert spaces $X = Y = L^2(0,1) \times L^2(0,1)$ with the usual inner product, and define the operator $L: X \to Y$ by 
$$ L \begin{bmatrix}
u \\
v 
\end{bmatrix}
= \begin{bmatrix}
\partial_x u \\
\partial_x v 
\end{bmatrix}. $$
Note that here $L$ is unbounded and closed, but not invertible. Thus, this operator satisfies the main assumption made for the continuous case.  We let $ \Pi_r^Y $ be the orthogonal projection onto $ Y_r = \mathrm{span}\{ L \varphi_k \}_{k=1}^r $, where $ \{ \varphi_k \} \subset X $ are the POD modes.

To approximate the solution of the PDE we used the interpolated coefficient finite element method with continuous piecewise linear basis functions from \cite{Wang15}, and \texttt{ode23s} from MATLAB for the time stepping scheme.  We approximated the solution using 100 equally spaced finite element nodes on the time interval $ \mathcal{O} = (0,10) $.  Increasing the number of finite element nodes gave similar results below.

For the POD computations, the solution values were approximated at each time step, $w(t_k)$, where $ w = [u, v]^T $, and a piecewise constant function in time was formed. The constant on each interval is given by the average of the solution at the current step and the solution at the next step, i.e., $0.5(w(t_{k+1})+w(t_k))$. Note that for this problem we can calculate the POD eigenvalues, POD modes, and the data approximation errors exactly. Thus, comparisons between the actual approximation errors and the error formulas can be made. 

In \Cref{Table:Error with r = 4,Table:Error with r = 12} we present the errors from the relevant projections considered in this paper for $ r = 4 $ and $ r = 12 $. Note that errors for projections involving the inverse mapping $ L^{-1} $ are not included since $ L $ is not invertible for this example.  In the tables, the actual error is the integral error measure and the error formula is the sum involving the POD singular values.  The first line in the tables represents computations for the known error result \eqref{knownPODerror_cont}.  The second and third lines of the tables are computations for the new results \eqref{error1}-\eqref{error2}.  The second line of each table gives the values for
$$
\text{actual error } = \int_{\mathcal{O}} \|Lw(t)-L\Pi_r^X w(t) \|_Y^2 dt,  \quad  \text{error formula } = \sum_{k > r} \sigma_{k}^2 \|L\varphi_k\|^2_Y,
$$
while the third line of each table shows computational results for
$$
\text{actual error } = \int_{\mathcal{O}} \|L w_j -  \Pi _r^Y L w_j \|^2_{Y} dt,  \quad  \text{error formula } = \sum_{k > r} \sigma_{k}^2 \|L \varphi_k -  \Pi _r^Y L \varphi \|^2_Y.
$$
The differences in the computed values are likely due to round off errors.  Note that as $ r $ increases the errors tend toward zero, as expected by the theory.
%
%
\begin{table}
	
	\renewcommand{\arraystretch}{1.25}
	\caption{Error Comparison with $r = 4$\label{Table:Error with r = 4}}
	\begin{center}
		\begin{tabular}{|c||c|c|c|}
			\hline
			POD Error Equation & Actual Error & Error Formula & Difference\\
			\hline
			\Cref{knownPODerror_cont}  & $6.2755 \times 10^{-5}$ & $6.2792 \times 10^{-5}$ &$3.7584\times 10^{-8}$\\
			\Cref{error1}  & $2.1584 \times 10^{-1}$ & $2.1593 \times 10^{-1}$& $9.1863 \times 10^{-5}$\\
			\Cref{error2}  & $9.8536 \times 10^{-3}$ & $9.8541 \times 10^{-3}$ & $4.7712 \times 10^{-7}$\\
			\hline
		\end{tabular}
	\end{center}
\end{table} 

\begin{table}
	
	\renewcommand{\arraystretch}{1.25}
	\caption{Error Comparison with $r = 12$\label{Table:Error with r = 12}}
	\begin{center}
		\begin{tabular}{|c||c|c|c|}
			\hline
			POD Error Formula & Actual Error & Error Formula & Difference\\
			\hline
			\Cref{knownPODerror_cont} & $4.1453 \times 10^{-8}$ & $4.1487 \times 10^{-8}$ &$3.3661\times 10^{-11}$\\
			\Cref{error1} & $2.2536 \times 10^{-4}$ & $2.2541 \times 10^{-4}$& $5.2146 \times 10^{-8}$\\
			\Cref{error2} & $1.2664 \times 10^{-5}$ & $1.2668 \times 10^{-5}$ & $3.5150 \times 10^{-9}$\\
			\hline
		\end{tabular}
	\end{center}
\end{table}

\section{POD Properties}
\label{Section:basicresults}

In this section, we consider three topics.  In \Cref{sec:HS_POD_operators}, we give two results concerning Hilbert-Schmidt operator norms of POD operators and approximations of POD operators.  These Hilbert-Schmidt results are used throughout \Cref{Section:errors} and \Cref{Section:convergence}.  In \Cref{sec:nonorth_POD_projections}, we study the boundedness of various non-orthogonal POD projections.  These boundedness results are used in \Cref{Section:convergence}.  In \Cref{sec:POD_singular_values}, we study the relationship between POD singular values and POD eigenvalues for different data.  This investigation is motivated by some results in \Cref{Section:convergence} where we assume the POD eigenvalues of different data are all nonzero.

\subsection{Hilbert-Schmidt Results for POD Operators}\label{sec:HS_POD_operators}

Below, we give two Hilbert-Schmidt results concerning POD operators.  The first result is known (see, e.e., \cite[Section 3.5]{Balakrishnan76}, \cite[Theorem 12.6.1]{Aubin00}, \cite[Lemma 4.4]{Singler11}), although perhaps not exactly in this precise form.  We provide a proof to be complete, and also since the result is crucial to this work.
%
%
\begin{lemma}\label{HS}
	Let $ Z $ be a separable Hilbert space, and let $ S = L^2(\mathcal{O};\mathbb{K}^m) $, where $ \mathcal{O} $ is an open subset of $ \mathbb{R}^d $.  If $K: S \to Z$ is defined by
	$$K f = \sum_{j=1}^{m} \int_\mathcal{O} \, f_j(t) z_j(t) \, dt,$$
	for $\{z_j\}_{j=1}^m \subset L^2(\mathcal{O};Z)$, then $K$ is Hilbert-Schmidt and 
	$$\|K\|^2_{\mathrm{HS}(S,Z)} = \sum_{j=1}^{m} \|z_j\|^2_{L^2(\mathcal{O};Z)}.$$
\end{lemma}

\begin{proof}
	Let $ \{ \chi_{i} \}_{i \geq 1} \subset L^2(\mathcal{O}) $ and $ \{ \xi_n \}_{n \geq 1} \subset Z $ be orthonormal bases.  Therefore, $ \{ \overline{\chi}_{i} \}_{i \geq 1} $ is also an orthonormal basis for $ L^2(\mathcal{O}) $, and $ \{ \chi_{i} \xi_n \}_{i,n \geq 1} $ is an orthonormal basis for $ L^2(\mathcal{O};Z) $ (see, e.g., \cite[Theorem 12.6.1]{Aubin00}).
	
	For $ \xi \in Z $, let $ [K^* \xi]_j = (\xi,z_j(t))_Z $ denote the $ j $th component of $K^*\xi \in S $.  Working with the Hilbert adjoint operator $ K^* $ and using Parseval's equality gives
	\begin{align*}\allowdisplaybreaks
	\| K^* \|_{\mathrm{HS}(S,Z)}^2  &=  \sum_{n \geq 1} \| K^* \xi_n \|^2_S\\
	&=  \sum_{j = 1}^m  \sum_{n \geq 1} \big\| [K^* \xi_n]_j \big\|^2_{L^2(\mathcal{O})}\\
	&=  \sum_{j = 1}^m  \sum_{n,i \geq 1} \left| \big( \overline{\chi}_{i}, [K^* \xi_n]_j \big)_{L^2(\mathcal{O})} \right|^2\\
	&=  \sum_{j = 1}^m  \sum_{n,i \geq 1} \left| \int_{\mathcal{O}} \overline{\chi_{i}(t)} \, ( z_j(t), \xi_n )_Z \, dt \right|^2\\
	&=  \sum_{j = 1}^m  \sum_{n,i \geq 1} \left| \int_{\mathcal{O}} ( z_j(t), \chi_{i}(t) \xi_n )_Z \, dt \right|^2\\
	&=  \sum_{j = 1}^m  \sum_{n,i \geq 1} \left| ( z_j, \chi_{i} \xi_n )_{L^2(\mathcal{O};Z)} \right|^2\\
	&=  \sum_{j = 1}^m  \| z_j \|_{L^2(\mathcal{O};Z)}^2.
	\end{align*}
\end{proof}

The next result gives three different Hilbert-Schmidt norm approximation results involving the POD operator $ K $ for the data $ \{ w_j\} $ and the POD operator $ K^Y = L K $ for the data $ \{ L w_j\} $.  The result will be of particular usefulness when discussing the continuous case in \Cref{Section:ContError}, but it applies to the discrete case as well.  We also use this result throughout \Cref{Section:convergence}.
%
%
\begin{lemma}\label{HSconv}  The Hilbert-Schmidt norm errors are given by
	\begin{equation}\label{HSnormLPiXK}
	\|LK-L \Pi _r^X K \|^2_{\mathrm{HS}(S,Y)} = \sum_{k>r} \sigma_k^2 \|L \varphi_k \|^2_Y,
	\end{equation}
	\begin{equation}\label{HSnormPiPsiLK}
	\|LK- \Pi _r^Y L K \|^2_{\mathrm{HS}(S,Y)} = \sum_{k>r} \sigma_k^2 \|L \varphi_k - \Pi _r^Y L \varphi_k \|^2_Y,
	\end{equation}
	and
	\begin{equation}\label{HSnormL^{-1}PiPsiLK}
	\| K - L^{-1} \Pi _r^Y L K \|^2_{\mathrm{HS}(S,X)} = \sum_{k>r} \sigma_k^2 \| \varphi_k - L^{-1} \Pi _r^Y L \varphi_k \|^2_X.
	\end{equation}
	In the case $s_X = \infty$, the following convergence results hold.  For \eqref{HSnormLPiXK}: The error tends to zero as $r \to \infty$.  For \eqref{HSnormPiPsiLK}: If $\{\Pi_r^Y\}$ is uniformly bounded in operator norm, then the error goes to zero as $ r \to \infty$. For \eqref{HSnormL^{-1}PiPsiLK}: If $L^{-1}$ is bounded and $\{\Pi_r^Y\}$ is uniformly bounded in operator norm, then the error tends to zero as $r \to \infty$.  For \eqref{HSnormL^{-1}PiPsiLK}: If $\{L^{-1} \Pi_r^Y L\}$ is uniformly bounded in operator norm, then the error converges to zero as $r \to \infty$. 
\end{lemma}
\begin{proof}
	Let $\{f_k\}$ be an orthonormal basis of $S$ of eigenvectors of $K^*K$ and let $\mathbb{J} = \{k:f_k \notin \mathrm{ker} (K^*K) \}$. Note that $Kf_k = 0$ for all $k\notin \mathbb{J}$, since $\mathrm{ker}(K^*K) = \mathrm{ker}(K)$ by \Cref{lem1}. Also, $Kf_k = \sigma_k \varphi_k$ for all $k\in \mathbb{J}$. Then,
	\begin{align*}
	\|LK- L \Pi _r^X K\|^2_{\mathrm{HS}(S,Y)} &= \sum_{k\geq 1} \|(LK- L \Pi _r^X K)f_k\|^2_{Y} \\
	& = \sum_{k \in \mathbb{J}} \|(LK- L \Pi _r^X K)f_k\|^2_{Y} \\
	& = \sum_{k \in \mathbb{J}} \| L\sigma_k \varphi_k - L\Pi _r^X \sigma_k \varphi_k \|^2_{Y} \\
	& = \sum_{k>r, \, k \in \mathbb{J}} \sigma_k^2 \|L\varphi_k\|^2_{Y},
	\end{align*}
	where the last equality holds since $\Pi _r^X \varphi_k = \varphi_k $ for $k\leq r$ and $\Pi _r^X \varphi_k = 0 $ for $k > r$.
	Also, 
	$$\sum_{k>r, \, k \in \mathbb{J}} \sigma_k^2 \|L\varphi_k\|^2_Y = \sum_{k>r, \, k \in \mathbb{J}} \|LKf_k\|^2_{Y} = \sum_{k>r, \, k \in \mathbb{J}} \|K^Yf_k\|^2_{Y},$$
	which converges to zero as $r\to \infty$ since $K^Y$ is Hilbert-Schmidt. 
	
	Next, 
	\begin{align*}
	\|LK - \Pi _r^Y LK\|^2_{\mathrm{HS}(S,Y)} &= \sum_{k\geq 1} \| (LK - \Pi _r^Y LK)f_k\|^2_{Y} \\
	& = \sum_{k\in \mathbb{J}} \| L\sigma_k\varphi_k - \Pi _r^Y L\sigma_k\varphi_k\|^2_{Y} \\
	& = \sum_{k>r, \, k \in \mathbb{J}} \sigma_k^2 \|L\varphi_k-\Pi _r^Y L \varphi_k \|^2_{Y},
	\end{align*}
	where the last equality holds since $\Pi _r^Y L \varphi_k = L\varphi_k $ for $k\leq r$. For convergence, note
	\begin{align*}
	\|LK - \Pi _r^Y LK\|^2_{\mathrm{HS}(S,Y)} &= \sum_{k>r, \, k \in \mathbb{J}} \|LKf_k-\Pi _r^Y LKf_k\|^2_{Y} \\ &= \sum_{k>r, \, k \in \mathbb{J}} \|K^Yf_k-\Pi _r^Y K^Yf_k\|^2_{Y} \\ &\leq \sum_{k>r, \, k \in \mathbb{J}} \|I-\Pi _r^Y\|^2 \, \|K^Y f_k\|^2_Y.
	\end{align*} 
	Since $\|I-\Pi _r^Y\|$ is uniformly bounded and $K^Y$ is Hilbert-Schmidt, the error converges to zero as $r\to\infty$. 
	
	Similarly, for the last equality we have
	\begin{align*}
	\| K - L^{-1} \Pi _r^Y L K \|^2_{\mathrm{HS}(S,X)} & = \sum_{k\in \mathbb{J}}\|\sigma_k\varphi_k - L^{-1}\Pi _r^Y L\sigma_k\varphi_k\|^2_{X}\\
	& = \sum_{k>r, \, k \in \mathbb{J}} \sigma_k^2 \|\varphi_k - L^{-1}\Pi _r^Y L\varphi_k\|^2_{X},
	\end{align*}
	since $L^{-1}\Pi _r^Y L\varphi_k = L^{-1}L\varphi_k=\varphi_k$ for $k \leq r$.
	
	Assuming $L^{-1}$ is bounded and $\{ \Pi_r^Y \}$ is uniformly bounded, the convergence follows from 
	\begin{align*}
	\| K - L^{-1} \Pi _r^Y L K \|^2_{\mathrm{HS}(S,X)} & = \sum_{k>r, \, k \in \mathbb{J}} \| L^{-1}(I-\Pi _r^Y)L K f_k\|^2_{X} \\
	& \leq \sum_{k>r, \, k \in \mathbb{J}} \|L^{-1}\|^2 \,\|I-\Pi _r^Y\|^2 \, \|K^Y f_k\|^2_{Y}
	\end{align*}
	in a similar manner to the previous case.  For the second convergence case, we assume $\{ L^{-1} \Pi_r^Y L \}$ is uniformly bounded in operator norm and we have
	\begin{align*}
	\| K - L^{-1} \Pi _r^Y L K \|^2_{\mathrm{HS}(S,X)} & = \sum_{k>r, \, k \in \mathbb{J}} \| (I-L^{-1}\Pi _r^Y L) K f_k\|^2_{X} \\
	& \leq \sum_{k>r, \, k \in \mathbb{J}} \|I-L^{-1}\Pi _r^Y L\|^2 \, \|K^Y f_k\|^2_X,
	\end{align*}
	which converges to zero as $r\to \infty$.

\end{proof}

\subsection{Non-orthogonal POD Projections}\label{sec:nonorth_POD_projections}

In \Cref{Section:convergence}, we consider pointwise convergence results for the linear operators $L^{-1}\Pi_r^Y L:X \to X$ and $L \Pi_r^X L^{-1}:Y \to Y$.  Below, we give conditions that guarantee that these linear operators are bounded, or have bounded extensions, for $ r $ fixed.  We note that when these operators are bounded we have $L^{-1}\Pi_r^Y L:X \to X$ is a projection onto $ X_r = \mathrm{span}\{ \varphi \}_{j=1}^r $ and $L \Pi_r^X L^{-1}:Y \to Y$ is a projection onto $ Y_r = \mathrm{span}\{ L\varphi \}_{j=1}^r $.  Even if $ \Pi_r^Y $ is an orthogonal projection, these projections are typically non-orthogonal POD projection operators.

In the simplest case, if $ L $ and $ L^{-1} $ are bounded, then clearly $L^{-1}\Pi_r^Y L:X \to X$ and $L \Pi_r^X L^{-1}:Y \to Y$ are both bounded for each $ r $.  In this case, $ \{ L \Pi_r^X L^{-1} \} $ is uniformly bounded in operator norm, and $ \{ L^{-1}\Pi_r^Y L \} $ is also uniformly bounded when $ \{ \Pi_r^Y \} $ is uniformly bounded.

Below, we consider the case when either $ L $ or $ L^{-1} $ is unbounded.  For each fixed $ r $, we show $L^{-1}\Pi_r^Y L:X \to X$ is bounded when $ L $ is bounded, and $L \Pi_r^X L^{-1}:Y \to Y$ is bounded when $ L^{-1} $ is bounded.  In other cases, we need certain assumptions to be satisfied to construct bounded extensions of the operators for each $ r $.  We do \textit{not} show that these non-orthogonal POD projection operators are uniformly bounded in operator norm.


In specific cases, we need certain adjoint operators to exist and therefore we need the operators to be densely defined or bounded.  For example, for the operator $L^{-*}$ to exist we must assume that $\mathcal{D}(L^{-1})$ is dense in $Y$, or $L^{-1}$ is bounded. These type of assumptions must be added to the second and fourth parts of the following theorem, in addition to results later in this paper.
\begin{theorem}\label{boundedext} 
	\begin{enumerate} Assume $L$ is invertible and $r>0$ is fixed. 
		\item If $L^{-1}$ is bounded, then $L\Pi_r^X L^{-1}:Y \to Y$ is bounded. 
		
		\item If $\mathcal{D}(L^{-1})$ is dense in $Y$ and $\{\varphi_k\}_{k=1}^r \subset \mathcal{D}(L^{-*})$, the operator $L\Pi_r^X L^{-1}: Y \to Y$ can be extended to a bounded operator on $Y$.
		
		\item  If $L$ is bounded, then $L^{-1}\Pi^Y_r L:X \to X $ is bounded. 
		
		\item Assume $ \Pi_r^Y : Y \to Y $ is the orthogonal projection onto $ \text{span}\{ L \varphi_k \}_{k=1}^r $.  If $L^{-1}$ is bounded, $\mathcal{D}(L)$ is dense, and $\{ L \varphi_k \}_{k=1}^r \subset \mathcal{D}(L^*)$, 
		then $L^{-1}\Pi^Y_r L:X \to X $ can be extended to a bounded operator on $X$. 
	\end{enumerate}
	
\end{theorem}

\begin{remark}
	In the second and fourth items, we assume the POD modes satisfy the regularity properties $\{\varphi_k\}_{k=1}^r \subset \mathcal{D}(L^{-*})$ and $\{ L \varphi_k \}_{k=1}^r \subset \mathcal{D}(L^*)$, respectively.  See \Cref{Section:mainassumptions} for conditions on the data in the discrete and continuous cases that guarantee these properties hold.
	%
\end{remark}

\begin{proof} 
	\begin{enumerate}
		%
		%
		\item Note
		\begin{equation}\label{above}
		L \Pi_r^X L^{-1} y = \sum_{k=1}^{r} (L^{-1}y,\varphi_k)_X L \varphi_k.
		\end{equation}
		Since $L^{-1}$ is a bounded operator and $\varphi_k \in \mathcal{D}(L)$ for all $k$, the sum in \eqref{above} is well defined for all $ y \in Y $. Also, it can be checked that
		\begin{equation}\label{above_bounded}
		\|L \Pi_r^X L^{-1} y \|_Y  \leq  c \|y\|_Y,
		\end{equation}
		%
		%
		where the constant $c := \|L^{-1}\| \left( \sum_{k=1}^{r} \|L\varphi_k\|_Y^2 \right)^{1/2}$ depends on $r$. This shows that the operator $L \Pi_r^X L^{-1}$ is bounded when $L^{-1}$ is bounded. 
		
		%
		%
		\item  The linear operator $ L \Pi_r^X L^{-1} : Y \to Y $ is defined by \eqref{above} for all $ y \in \mathcal{D}(L^{-1}) $.  Using the assumptions, we can rewrite \eqref{above} for $ y \in \mathcal{D}(L^{-1}) $ as
		\begin{equation}\label{above2}
		L \Pi_r^X L^{-1} y = \sum_{k=1}^{r} (y,L^{-*}\varphi_k)_X L \varphi_k.
		\end{equation}
		It can be checked that \eqref{above_bounded} holds for all $ y \in \mathcal{D}(L^{-1}) $ with $ c := \sum_{k=1}^r \| L^{-*}\varphi_k \|_X \| L \varphi_k \|_Y $.  Note that \eqref{above2} is well-defined for all $ y \in Y $, and therefore yields a bounded linear extension of $ L \Pi_r^X L^{-1} : Y \to Y $ to all of $ Y $.  
		
		
		%
		%
		\item\label{boundedext_proof_part3} Since $\Pi_r^Y$ is a projection onto $ Y_r = \mathrm{span} \{ L\varphi \}_{j=1}^r $, we know for $y \in Y$ there exists constants $\{\alpha_j(y)\}$ depending on $ y $ such that $\Pi_r^Y y = \sum_{j=1}^{r} \alpha_j(y) L\varphi_j$. Then
		\begin{equation}\label{Pi_r^Psi bounded}
		\bigg\|\sum_{j=1}^{r} \alpha_j(y) L\varphi_j \bigg\|_Y = \|\Pi_r^Y y \|_Y \leq \|\Pi_r^Y\| \|y\|_Y.
		\end{equation}
		Also, 
		$$\bigg\| \sum_{j=1}^{r} \alpha_j(y) L\varphi_j \bigg\|_Y^2 = \sum_{j,k = 1}^r \alpha_j(y) (L\varphi_j,L\varphi_k)_Y \overline{\alpha_k(y)} = \alpha(y)^* A_r \alpha(y),$$
		where the star denotes complex conjugate, and
		$$\alpha(y) = [\alpha_1(y),...,\alpha_r(y)]^T\in \mathbb{K}^r,  \quad  [A_r]_{i,j} = (L\varphi_i,L\varphi_j)_Y.$$
		Since $ L $ is invertible and $ \{ \varphi_j \}_{j=1}^r $ is a linearly independent set, we know $\{L\varphi_j\}_{j=1}^r$ is a linearly independent set; therefore, $A_r$ is symmetric positive definite, which implies there exists $\beta >0$ such that $\alpha^* A_r \alpha \geq \beta \|\alpha\|_{\mathbb{K}^r}^2$ for all $ \alpha \in \mathbb{K}^r $. Note that $\beta$ may depend on $r$. Together, the above implies that
		\begin{equation*}
		\beta \|\alpha(y)\|_{\mathbb{K}^r}^2 \leq \bigg\| \sum_{j=1}^{r} \alpha_j(y) L\varphi_j \bigg\|_Y^2 
		\leq \|\Pi_r^Y\|^2 \|y\|_Y^2.
		\end{equation*}
		So, 
		\begin{equation}
		\|\alpha(y)\|_{\mathbb{K}^r} \leq \beta^{-1/2} \|\Pi_r^Y\| \|y\|_Y.
		\end{equation}
		
		In this case, $y = Lx$ and $L$ is bounded and invertible; thus,
		$$ L^{-1} \Pi_r^Y (Lx) = L^{-1} \sum_{j = 1}^r \alpha_j(Lx) L\varphi_j = \sum_{j = 1}^r \alpha_j(Lx) \varphi_j$$
		where the constants $\alpha_j$ now depend on $Lx$. Since $\{\varphi_j\} \subset X$ is orthonormal, we have
		\begin{align*}
		\|L^{-1} \Pi_r^Y Lx\|_X^2 &= \bigg\|\sum_{j = 1}^r \alpha_j(Lx) \varphi_j \bigg\|_X^2 \\
		& = \|\alpha(Lx)\|_{\mathbb{K}^r}^2 \\
		& \leq \beta^{-1} \|\Pi_r^Y\|^2 \|Lx\|_Y^2\\
		& \leq \beta^{-1} \|\Pi_r^Y\|^2 \|L\|^2 \|x\|_X^2.
		\end{align*}
		Therefore, for all $ x \in X $ we have
		\begin{equation}\label{bounded_part3}
		\|L^{-1} \Pi_r^Y L x\|_X \leq  c \| x \|_X,
		\end{equation}
		where $ c := \beta^{-1/2} \|\Pi_r^Y\| \|L\| $.
		
		%
		%
		\item We obtain a representation of $L^{-1} \Pi_r^Y L $ as follows.  First, note that the sets $\{L\varphi_k\}$ and $\{ L^{-*} \varphi_k \} $ are biorthogonal, i.e., $(L\varphi_k, L^{-*}\varphi_j)_Y = \delta_{k,j}$, where $ \delta_{k,j}$ is the Kronecker delta symbol. Recall from the proof of part \ref{boundedext_proof_part3} that $\Pi_r^Y y = \sum_{k=1}^{r} \alpha_k L\varphi_k$ for some scalars $\alpha_k $ that depend on $ y $. We can calculate the values for $\alpha_k$ by noting
		$$\left( \Pi_r^Y y, L^{-*}\varphi_j \right)_Y = \sum_{k=1}^{r} \alpha_k(L\varphi_k, L^{-*}\varphi_j)_Y = \alpha_j.$$
		This yields
		\begin{equation}\label{psiprojectionequ}
		\Pi_r^Y y = \sum_{k=1}^{r} ( \Pi_r^Y y, L^{-*}\varphi_k)_Y L\varphi_k  = \sum_{k=1}^{r} (y, \Pi_r^Y L^{-*}\varphi_k)_Y L\varphi_k,
		\end{equation}
		since $\Pi_r^Y$ is orthogonal and therefore $(\Pi_r^Y)^* = \Pi_r^Y$.
		
		By assumption, $ \{ L \varphi_j \} \subset \mathcal{D}(L^*) $ and so \eqref{psiprojectionequ} implies $ \Pi_r^Y y \in \mathcal{D}(L^*) $ for all $ y \in Y $.  This gives the following representation for any $ x \in \mathcal{D}(L) $:
		\begin{equation}\label{extension_part4}
		L^{-1} \Pi_r^Y Lx = \sum_{k=1}^{r} (x, L^* \Pi_r^Y L^{-*} \varphi_k)_X  \varphi_k.
		\end{equation}
		Also, for all $ x \in \mathcal{D}(L) $, the bound \eqref{bounded_part3} holds with $ c := \left( \sum_{k=1}^{r} \| L^* \Pi_r^Y L^{-*} \varphi_k \|_X^2 \right)^{1/2} $.  Equation \eqref{extension_part4} is well-defined for all $ x \in X $, and therefore defines a bounded linear extension of $ L^{-1} \Pi_r^Y L : X \to X $ to all of $ X $.
		%
	\end{enumerate}
\end{proof}


\subsection{POD Singular Values and POD Eigenvalues}\label{sec:POD_singular_values}


The number of nonzero singular values (or eigenvalues) of the POD operators plays an important role throughout the paper. It is also important to note the difference between singular values and eigenvalues. For a POD operator $K:S\to Z$, recall the POD eigenvalues are the eigenvalues of $KK^*:Z\to Z$, the POD singular values are the singular values of $K$, and $ s_Z = \mathrm{rank}(K) $, i.e., $ s_Z $ is the number of positive POD singular values of $ K $ (or positive POD eigenvalues of $ K K^* $).  As discussed in \Cref{Section:basic}, it is possible to have a zero POD singular value but to have all nonzero POD eigenvalues.

Below, we study various relationships between the POD eigenvalues and POD singular values for the data $\{w_j\}$ and the data $\{Lw_j\}$.  Recall, $ K : S \to X $ is the POD operator for the data $\{w_j\}$, and $ K^Y = LK : S \to Y $ is the POD operator for the data $ \{ L w_j \} $.  Therefore, $ s_X = \mathrm{rank}(K) $ is the number of nonzero POD singular values (or POD eigenvalues) for the data $\{w_j\}$, and $ s_Y = \mathrm{rank}(K^Y) $ is the number of nonzero POD singular values (or POD eigenvalues) for the data $\{ Lw_j\}$

First, we give a relationship between the POD eigenvalues and the null space of the adjoint POD operator, and also give some additional information about $ s_X $ and $ s_Y $.
\begin{lemma}\label{prob4}
	\begin{enumerate}
		\item\label{prob4_part1} All of the POD eigenvalues for the data $\{w_j\} $ are nonzero if and only if $\mathrm{ker}(K^*) = \{0\} $.  In this case, $ X = \overline{ \mathcal{R}(K) } $.  In addition, if $s_X<\infty$, then $ X = \mathcal{R}(K) $ and $ \dim(X) = s_X $.
		
		\item\label{prob4_part2} All of the POD eigenvalues for the data $ \{ L w_j \} $ are nonzero if and only if $ \mathrm{ker}((K^Y)^*) = \{0\} $.  In this case,  $ Y = \overline{ \mathcal{R}(K^Y) } $.  In addition, if $s_Y<\infty$, then $ Y = \mathcal{R}(K^Y) $ and $ \dim(Y) = s_Y $.
		
		\item The number of nonzero POD eigenvalues for $\{Lw_j\}$ is less than or equal to the number of nonzero POD eigenvalues for $\{w_j\}$. That is, $s_Y \leq s_X$.
		
		\item If $L$ is invertible, then $s_X =s_Y$. 
	\end{enumerate}
\end{lemma}

\begin{proof}
	The first two items are proven similarly. Here we show item 1. 	
	\begin{enumerate}
		\item 
		
		\Cref{lem1} proves the first statement.  To see the rest, note that $X = \mathrm{ker}(K^*) \oplus \overline{ \mathcal{R}(K) }$ and $\mathrm{ker}(K^*) = \{0\} $ imply $ X = \overline{ \mathcal{R}(K) } $.  Then if $ s_X = \mathrm{rank}(K) = \dim ( \mathcal{R}(K) )$ is finite, we have $\overline{ \mathcal{R}(K) } =  \mathcal{R}(K)$ and therefore $X = \mathcal{R}(K) $ and $ \dim(X) = s_X $.
		
		
		
		\item[3.] First, if $s_X = \infty$, we are done. Assume $s_X < \infty$. We know
		$$ K f = \sum_{j=1}^{s_X} \sigma_j (f,f_j)_S \varphi_j, $$
		and therefore
		$$
		K^Y f = LK f = \sum_{j=1}^{s_X} \sigma_j (f,f_j)_S L \varphi_j.
		$$
		Thus, $ s_Y = \mathrm{rank}(K^Y) \leq s_X $.
		%
		
		\item[4.] Because of item 3, we need only show $ s_X \leq s_Y $.  First, if $s_Y = \infty$, we are done. Assume $s_Y < \infty$. Let the singular value decomposition of $ K^Y $ be given by
		$$K^Y f = LKf = \sum_{j=1}^{s_Y} \sigma_j^Y (f,f_j^Y)_S \varphi_j^Y. $$
		Note that $\varphi_j^Y \in \mathcal{D}(L^{-1})$ whenever $\sigma_j^Y > 0$, since $\mathcal{D}(L^{-1})
		= R(L)$ and $\varphi_j^Y = (\sigma_{j}^Y)^{-1}K^Yf_j^Y = (\sigma_{j}^Y)^{-1}LKf_j^Y$. Then, since $L$ is invertible,
		$$Kf = L^{-1}LKf = L^{-1}K^Yf = \sum_{j=1}^{s_Y} \sigma_j^Y (f,f_j^Y)_S L^{-1}\varphi_j^Y, $$
		and therefore $ s_X = \mathrm{rank}(K) \leq s_Y $.
		%
		%
		%
		%
	\end{enumerate}
\end{proof}

The following lemma gives further results about the connections between the two main sets of POD eigenvalues under consideration in this paper, i.e., the POD eigenvalues for the data $\{w_j\}$ and the data $\{Lw_j\}$. With extra assumptions, we can use the fact that all the POD eigenvalues are nonzero for one set of data to obtain the same conclusion for the other set of data. 

\begin{lemma} \label{xypodrel}
	\begin{enumerate}
		\item If $L$ is bounded, $\mathcal{R}(L)$ is dense in $Y$, and the POD eigenvalues for $\{w_j\}$ are all nonzero, then the POD eigenvalues for $\{Lw_j\}$ are all nonzero. 
		
		\item If $L^{-1}$ is bounded, $\mathcal{R}(L^{-1})$ is dense in $X$, and the POD eigenvalues for $\{Lw_j\}$ are all nonzero, then the POD eigenvalues for $\{w_j\}$ are all nonzero.
	\end{enumerate} 
\end{lemma}

\begin{proof}
	
	The proofs of the two items are similar; we only prove the first item.
	
	Since $X = \mathrm{ker}(K^*) \oplus \overline{ \mathcal{R}(K) }$ and $\mathrm{ker}(K^*) = \{0\} $ (\Cref{prob4}, \Cref{prob4_part1}), we have $ X = \overline{ \mathcal{R}(K) } $.  Let $ \varepsilon > 0 $ and let $ y \in Y $.  Since $ \mathcal{R}(L) $ is dense in $ Y $, there exists $x \in X$ such that $\|y-Lx\|_Y < \varepsilon/2$. Since $X = \overline{\mathcal{R}(K)}$, for  this $x$ there exists $f \in S$ such that $\|x-Kf\|_X < \varepsilon/(2\|L\|) $.  This gives
	%
	%
	$$\|y-LKf\|_Y < \|y-Lx\|_Y+\|Lx-LKf\|_Y < \frac{\varepsilon}{2} + \|L\|\frac{\varepsilon}{2\|L\|} < \varepsilon,$$
	which shows $\overline{\mathcal{R}(K^Y)}=Y$ and $\mathrm{ker}((K^Y)^*) = \{0\}$.  Thus, the POD eigenvalues for $\{Lw_j\}$ are all nonzero by \Cref{prob4}, \Cref{prob4_part2}.
	%
\end{proof}

\section{Error Formulas}
\label{Section:errors}
One goal of this paper is to provide exact formulas for POD data approximation errors. The two main results of this section can be found in \Cref{discrete1} and \Cref{conterror}. The section is split between the discrete case, where we can use a more direct proof approach, and the continuous case, which requires more care since the data can have infinitely many nonzero POD eigenvalues. 

\subsection{Discrete Case}\label{discreteerror}


First we introduce several representations that will be useful in the proof of \Cref{discrete1} below.  Recall, $ s_X = \mathrm{rank}(K) < \infty $ is the number of nonzero POD singular values (or POD eigenvalues) for the data $\{w_j\}$.  By the known POD error formula \eqref{knownPODerror}, we have
\begin{equation*}
w_j = \Pi_{s_X}^X w_j = \sum_{k=1}^{s_X} \iprod{w_j}{\varphi_k}{X} \varphi_k  \quad  \text{ and }  \quad  Lw_j = \sum_{k=1}^{s_X} \iprod{w_j}{\varphi_k}{X} L\varphi_k.
\end{equation*}
Note that since the sums are finite, $ \{ \varphi_k \} \subset \mathcal{D}(L) $, and $L$ is linear we can pull $L$ through the sums in this section without any additional assumptions. This is one point where the discrete and continuous cases differ.

Next, from \Cref{discretebackground} we know for all $ j \leq s $ and $ k \leq s_X $ we have
$$
\iprod{w_j}{\varphi_k}{X} = \overline{ \iprod{\varphi_k}{w_j}{X} } = \overline{(K^*\varphi_k)_j} = \sigma_k \overline{f_{k,j}},
$$
where $ f_{k,j} $ denotes the $ j $th component of the singular vector $ f_k \in \mathbb{K}^s $.  This gives
\begin{equation}\label{wj and Lwj}
w_j = \sum_{k=1}^{s_X} \sigma_k \, \overline{f_{k,j}} \, \varphi_k  \quad  \text{ and }  \quad Lw_j = \sum_{k=1}^{s_X} \sigma_k \, \overline{f_{k,j}} \,  L\varphi_k.
\end{equation} 
%
Also, recall $\{f_j\}$ are orthonormal in $S$, which yields $$\sum_{j=1}^{s}\gamma_j \overline{f_{k,j}}{f_{\ell,j}} = (f_\ell,f_k)_{S} = \delta_{\ell,k}.$$

\begin{theorem} \label{discrete1} The data approximation errors are given by
	\begin{equation}
	\sum_{j=1}^{s} \gamma_j \|Lw_j-L\Pi _r^Xw_j\|_Y^2 = \sum_{k=r+1}^{s_X}\sigma_k^2\|L\varphi_k\|_Y^2,
	\end{equation}
	and
	\begin{equation}\label{2nddiscrete}
	\sum_{j=1}^{s}\gamma_j\|Lw_j-\Pi _r^{Y}Lw_j\|_Y^2=\sum_{k=r+1}^{s_X}\sigma_{k}^2 \|L\varphi_k-\Pi _r^Y L\varphi_k\|^2_Y.
	\end{equation}
	Also, if $L$ is invertible, then
	\begin{equation}
	\sum_{j=1}^{s}\gamma_j\|w_j-L^{-1}\Pi _r^{Y}Lw_j\|_X^2 = \sum_{k=r+1}^{s_X}\sigma_{k}^2 \|\varphi_k-L^{-1}\Pi _r^Y L\varphi_k\|^2_X .
	\end{equation} 
\end{theorem}

%
%
%
%


\begin{proof} We only prove \eqref{2nddiscrete}. The proofs of the other two results are similar.	First, note we can apply $\Pi_r^Y$ to $Lw_j$ given in \eqref{wj and Lwj} to get
	%
	%
	$$Lw_j -\Pi _r^Y w_j = \sum_{k=1}^{s_X} \sigma_k \, \overline{f_{k,j}} \, (L\varphi_k-\Pi _r^Y L\varphi_k). $$
	Then
	\begin{align*}
	\sum_{j=1}^{s}\gamma_j\|Lw_j -\Pi _r^Y Lw_j\|^2_Y &= \sum_{j=1}^{s}\gamma_j \left(\sum_{k=1}^{s_X} \sigma_k \overline{f_{k,j}} (L\varphi_k-\Pi _r^Y L\varphi_k),  \sum_{\ell=1}^{s_X} \sigma_\ell \overline{f_{\ell,j}} (L\varphi_\ell-\Pi _r^Y L\varphi_\ell) \right)_Y \\
	& = \sum_{j=1}^s\gamma_j \sum_{\ell,k=1}^{s_X} \sigma_k \sigma_\ell \overline{f_{k,j}}{f_{\ell,j}} \left( L\varphi_k-\Pi _r^Y L\varphi_k, L\varphi_\ell-\Pi _r^Y L\varphi_\ell \right)_Y \\
	& = \sum_{\ell,k=1}^{s_X} \sigma_k \sigma_\ell \left(\sum_{j=1}^s \gamma_j \overline{f_{k,j}}{f_{\ell,j}} \right) \left( L\varphi_k-\Pi _r^Y L\varphi_k, L\varphi_\ell-\Pi _r^Y L\varphi_\ell \right)_Y \\
	& = \sum_{k=1}^{s_X} \sigma_k^2 \left( L\varphi_k-\Pi _r^Y L\varphi_k, L\varphi_k-\Pi _r^Y L\varphi_k \right)_Y \\
	& =\sum_{k=1}^{s_X} \sigma_k^2 \|L\varphi_k-\Pi _r^Y L\varphi_k\|^2_Y.
	\end{align*}
	Now note that $\Pi _r^Y L \varphi_k = L\varphi_k$ for $k=1,...,r$ since $ \Pi_r^Y $ is a projection onto $ Y_r = \mathrm{span}\{ L \varphi_k \}_{k=1}^r $.
	Therefore, 
	$$\sum_{j=1}^{s}\gamma_j\|Lw_j -\Pi _r^Y Lw_j\|^2_Y = \sum_{k=r+1}^{s_X}\sigma_{k}^2 \|L\varphi_k-\Pi _r^Y L\varphi_k\|^2_Y.$$
\end{proof}

\begin{remark}
	In \Cref{cor:error_bounds_each_snapshot}, we focus on error bounds for approximating each individual data snapshot $ w_\ell $ with various POD projections.  Also, another way to prove \Cref{discrete1} is to use the Hilbert Schmidt norm results in \Cref{HSconv}. The proof we give above requires less background. However, we do require \Cref{HSconv} for the continuous case below.
\end{remark}


%
%
%
%

\subsection{Continuous Case}
\label{Section:ContError}


For the continuous case we must consider the possibility that the number of nonzero POD eigenvalues is infinite.  We approach this case differently from the discrete case above. We show each of the data approximation errors we consider is equal to one of the Hilbert-Schmidt norm errors from \Cref{HSconv}. We use that result to prove the convergence of the errors to zero in the case of an infinite number of nonzero POD eigenvalues.

For one case, we need to make an additional assumption on $L^{-1}$.
\begin{quote}
	\textbf{The $L^{-1}$ assumption:} We assume 
	\begin{enumerate}
		\item $s_X < \infty$, or
		\item $
		L^{-1}\Pi_r^Y L K f = \sum_{j=1}^m \int_{\mathcal{O}} f_j(t) L^{-1}\Pi_r^Y Lw_j (t)dt
		$ for all $ f \in S $.
	\end{enumerate}
\end{quote}
\begin{remark}
	Note that if $s_X < \infty$, then the proof technique in \Cref{discreteerror} above can be used for the continuous cases, with some minor modifications to deal with the change in the space $ S $.  The second condition is similar to the main assumption made in \Cref{contcase:mainassumption}.  Any of the three common conditions in \Cref{remark:cont_assumption} that guarantee the main assumption holds also imply that the second condition in the $ L^{-1} $ assumption holds.
	%
\end{remark}

%

\begin{theorem}\label{conterror} The data approximation errors are given by
	\begin{equation}\label{conterroreq1}
	\sum_{j=1}^{m} \|L w_j - L \Pi _r^X w_j \|^2_{L^2(\mathcal{O};Y)}  = \sum_{k>r} \sigma_k^2 \|L \varphi_k \|^2_Y
	\end{equation}
	and	
	\begin{equation}\label{conterroreq2}
	\sum_{j=1}^{m} \|L w_j -  \Pi _r^Y L w_j \|^2_{L^2(\mathcal{O};Y)}  = \sum_{k>r} \sigma_k^2 \|L \varphi_k - \Pi _r^Y L \varphi_k \|^2_Y .
	\end{equation}
	Also if the $L^{-1}$ assumption holds then
	\begin{equation}\label{last cont error}
	\sum_{j=1}^{m} \| w_j - L^{-1} \Pi _r^Y L w_j \|^2_{L^2(\mathcal{O};X)}  = \sum_{k>r} \sigma_k^2 \| \varphi_k - L^{-1} \Pi _r^Y L \varphi_k \|^2_X .
	\end{equation}
	In the case $s_X = \infty$, the following convergence results hold.  For \eqref{conterroreq1}: The error tends to zero as $r \to \infty$.  For \eqref{conterroreq2}: If $\{\Pi_r^Y\}$ is uniformly bounded in operator norm, then the error goes to zero as $ r \to \infty$. For \eqref{last cont error}: If $L^{-1}$ is bounded and $\{\Pi_r^Y\}$ is uniformly bounded in operator norm, then the error tends to zero as $r \to \infty$.  For \eqref{last cont error}: If $\{L^{-1} \Pi_r^Y L\}$ is uniformly bounded in operator norm, then the error converges to zero as $r \to \infty$. 
	%
	
\end{theorem}

\begin{remark}
	Note that for the case $ s_X = \infty $, the conditions for convergence are exactly the conditions given in \Cref{HSconv}.
\end{remark}

\begin{proof} We prove \eqref{conterroreq1}, and the associated convergence result. The proofs of the other equalities and convergence results are similar.  We first show that the data approximation error has an integral representation, and then we use the two Hilbert-Schmidt results for POD operators from \Cref{sec:HS_POD_operators} to conclude.
	
	By definition, for $ f \in S $ we have
	\begin{align*}
	L\Pi _r^X Kf 
	& = \sum_{k=1}^{r}(Kf,\varphi_k)_X L \varphi_k \\
	& = \sum_{k=1}^{r}\left( \sum_{j=1}^m \int_\mathcal{O} f_j(t) w_j(t) dt ,\varphi_k\right)_X L \varphi_k \\
	& = \sum_{j=1}^m \int_\mathcal{O} f_j(t) \sum_{k=1}^r (w_j(t),\varphi_k)_X L \varphi_k dt \\
	& = \sum_{j=1}^m \int_\mathcal{O} f_j(t) Lw_j^r(t) dt,
	\end{align*}
	where $w_j^r(t)= \Pi _r^X w_j(t) = \sum_{k=1}^r (w_j(t),\varphi_k)_X  \varphi_k$. Because of the main assumption, we can pull the operator $ L $ inside the integral to give
	$$
	  (LK-L\Pi _r^X K)f = \int_\mathcal{O} \sum_{j=1}^m f_j(t) [Lw_j(t)-Lw_j^r(t)] dt.
	$$
	Since $ Lw_j-Lw_j^r \in L^2(\mathcal{O};Y) $ for each $ j $, by \Cref{HS} we have 
	$$\sum_{j=1}^{m} \|L w_j - L \Pi _r^X w_j \|^2_{L^2(\mathcal{O};Y)}  = \| LK - L \Pi _r^X K\|^2_{\mathrm{HS}(S,Y)}.$$
	\Cref{HSconv} proves both \eqref{conterroreq1} and the convergence result in the case $ s_X = \infty $.

	
	
	%
	Note for \eqref{last cont error}, for $ f \in S $ the $L^{-1}$ assumption gives
	\begin{equation} \label{need}
	L^{-1} \Pi _r^Y L K f = \int_\mathcal{O} \sum_{j=1}^m f_j(t) L^{-1} \Pi _r^Y L w_j(t) dt,
	\end{equation}
	and then we proceed similarly to establish the result.
\end{proof}




\section{Pointwise Convergence of POD Projections}
\label{Section:convergence}

Recall that $ \{ \varphi_k \} $ is an orthonormal basis for $ X $, and therefore $\|\Pi_r^X x - x\|_X \to 0$ for all $ x \in X $.  In this section, we prove various types of pointwise convergence results for the other POD projections; namely, $ \Pi_r^Y $ from \Cref{sec:overview}, and $ L \Pi_r^X L^{-1} $ and $ L^{-1} \Pi_r^Y L $ from \Cref{sec:nonorth_POD_projections}.  The majority of this section is not split into the discrete and continuous cases because the proofs are similar for both, and many of the results hold regardless of case.  We do focus on the discrete case at the end of this section and address some assumptions made in the literature about approximations of each individual data snapshot using POD projections.

Pointwise convergence results for these POD projections are easiest to obtain when $ L $ and $ L^{-1} $ are both bounded.  We primarily focus on the case when either $ L $ or $ L^{-1} $ is unbounded.

Range conditions are an important factor in this section. When an element to be approximated by a POD projection is in the range of $K$ or $K^Y$, we can often get better results.  When certain conditions hold, we know these ranges exactly.  Recall from \Cref{prob4}, if all the POD eigenvalues for $\{w_j\}$ are nonzero and $ s_X < \infty $, then we know $ X = \mathcal{R}(K) $ and $ \mathrm{dim}(X) = s_X $. Note that in this case, the Hilbert space $ X $ must be finite dimensional.  If all the POD eigenvalues for $\{w_j\}$ are nonzero and $ s_X = \infty $ (i.e., $ X $ must be infinite dimensional), then \Cref{prob4} only gives $ X = \overline{ \mathcal{R}(K) } $.  We do not always obtain the better convergence results in this case.  Similar statements hold for the spaces $Y$ and $\mathcal{R}(K^Y)$.

Also, as in \Cref{Section:errors}, we sometimes need to consider different proof techniques in the case $ s_X = \infty $.

We begin with a pointwise convergence result for $ \Pi_r^Y $ assuming $ L $ is bounded.  For another pointwise convergence result for $ \Pi_r^Y $ with different assumptions, see \Cref{contyconvergence} below.
\begin{theorem}\label{bounded convergence of Pi_r^Psi}
	Assume $ L $ is bounded and $\{ \Pi_r^Y \} $ is uniformly bounded in operator norm.  If $ y \in \mathcal{R}(L) $, then $\Pi_r^Y y \to y$ as $ r $ increases.  In addition, if $\mathcal{R}(L)$ is dense in $Y$, then $\Pi_r^Y y \to y$ for all $y \in Y$.
	%
\end{theorem}
\begin{proof}
	Let $y \in \mathcal{R}(L)$, so that $ y = L x $ for some $ x \in X $.  Note that since $ L \Pi_r^X x \in Y_r = \mathrm{span}\{ L \varphi_k \}_{k=1}^r $ and $ \Pi_r^Y $ is a projection onto $ Y_r $, we have $ \Pi_r^Y L \Pi_r^X x = L \Pi_r^X x $.  Then
	\begin{align*}
	\|\Pi_r^Y y - y \|_Y & \leq \| \Pi_r^Y L x - \Pi_r^Y L \Pi_r^X x \|_Y + \| \Pi_r^Y L \Pi_r^X x - L x \|_Y \\
	& = \| \Pi_r^Y L x - \Pi_r^Y L \Pi_r^X x \|_Y + \| L \Pi_r^X x - L x \|_Y \\
	& \leq \| \Pi_r^Y L \| \| x - \Pi_r^X x \|_Y + \| L \| \| \Pi_r^X x - x \|_Y,
	\end{align*}
	which converges to zero as $ r $ increases since $ \Pi_r^X x \to x $ and $\{ \Pi_r^Y \} $ is uniformly bounded in operator norm.  The final result follows directly from the Banach-Steinhaus theorem (i.e., the principle of uniform boundedness).
	%
	%
\end{proof}

The next convergence result relies on the boundedness of either $L$ or $L^{-1}$ and certain range conditions involving $ L $.
\begin{theorem} \label{convergencewithboundedop}
	\begin{enumerate}
		\item For any $ y \in \mathcal{R}(L) = \mathcal{D}(L^{-1}) $, if $L$ is bounded, then $ \| L \Pi_r^X L^{-1} y - y \|_Y \to 0 $ as $ r $ increases.  In addition, if $\mathcal{R}(L)$ is dense in $Y$ and $ \{ L \Pi_r^X L^{-1} \} $ is uniformly bounded, then $ L \Pi_r^X L^{-1} y \to y$ for all $y \in Y$.
		
		\item For any $x \in \mathcal{D}(L) = \mathcal{R}(L^{-1})$, if $L^{-1}$ is bounded and $\Pi_r^Y y\to y$ for all $y \in Y$ as $r$ increases, then $ \| L^{-1} \Pi_r^Y L x - x \|_X  \to 0$ as $r$ increases.  In addition, if $\mathcal{D}(L)$ is dense in $X$ and $ \{ L^{-1} \Pi_r^Y L \} $ is uniformly bounded, then $ L^{-1} \Pi_r^Y L x \to x$ for all $x \in X$.
	\end{enumerate}
\end{theorem}
\begin{remark}
	Note that \Cref{bounded convergence of Pi_r^Psi} and \Cref{contyconvergence} give two cases where the assumption $\Pi_r^Y y \to y$ for all $y \in Y$ holds.  Also, the uniform boundedness of $ \{ L \Pi_r^X L^{-1} \} $ and $ \{ L^{-1} \Pi_r^Y L \} $ is not currently known, unless $ L $ and $ L^{-1} $ are both bounded.  Note that when $ L $ and $ L^{-1} $ are both bounded, \Cref{bounded convergence of Pi_r^Psi} gives $ \Pi_r^Y y \to y $ for all $ y \in Y $ whenever $\{ \Pi_r^Y \} $ is uniformly bounded; therefore, in this case \Cref{convergencewithboundedop} gives $ L \Pi_r^X L^{-1} y \to y$ for all $y \in Y$ and $ L^{-1} \Pi_r^Y L x \to x$ for all $x \in X$.
\end{remark}
\begin{proof}  We only prove the first result; the proof of the second is similar.  Since $y\in \mathcal{R}(L)$ we have $y = Lx$ for some $x\in X$. Then
	\begin{align*}
	\| L \Pi_r^X L^{-1} y - y \|_Y &= \|L\Pi_r^X x -Lx\|_Y \\
	& \leq \|L\|\,\|\Pi_r^X x - x \|_X, 
	\end{align*}
	which converges to zero as $r$ increases.  The final convergence result again follows from the principle of uniform boundedness.
	%
	%
	%
\end{proof}


Next, we consider how range conditions involving $ K $ and $ K^Y $ affect the convergence of POD projections.  We are able to obtain convergence rates, and at most require either $ L $ or $ L^{-1} $ to be bounded.  We begin with the POD projection $ \Pi_r^Y $ and then consider $ L \Pi_r^X L^{-1} $ and $ L^{-1} \Pi_r^Y L $.  We use the following simple lemma multiple times below.
\begin{lemma}\label{lemma:yN_simple_lemma}
	Assume $y \in \mathcal{R}(K^Y)$ so that $ y = K^Y g = L K g $ for some $ g \in S $.  If
	\begin{equation}\label{yN}
	y_N = LK_N g = L \Pi_N^X Kg,
	\end{equation}
	then $ y_N \to y $ as $ N $ increases.
\end{lemma}
\begin{proof}
	As $N $ increases,
	$$ \|y_N - y \|_Y = \|LKg - LK_N g \|_Y \leq \|LK -  L \Pi_N^XK\|_{\mathrm{HS}(S,Y)} \|g\|_S \to 0$$
	by \Cref{HSconv}.
\end{proof}

%

Recall from \Cref{prob4} that $s_Y$ is always less than or equal to $s_X$. Thus if we assume $s_X < \infty$, we know that $s_Y < \infty$. For the following proofs, we consider whether $s_X$ is finite or infinite. 
\begin{theorem}\label{contyconvergence}
	Assume $ \{\Pi_r^Y\}$ is uniformly bounded in operator norm whenever $ s_X = \infty $.  If $y = K^Y g$ for some $g \in S$, then $ \Pi _r^Y y \to y $ as $ r $ increases and the following error bound holds:
	\begin{equation}\label{eqn:PirY_error_bound}
	\|\Pi_r^Y y - y\|_Y \leq \sum_{k>r} \sigma_k |(g,f_k)_{S}| \, \|\Pi _r^Y L \varphi_k - L \varphi_k\|_Y.
	%
	\end{equation}
	Also, if the POD eigenvalues for the data $ \{ L w_j \} $ are all nonzero, then $ \Pi _r^Y y \to y $ for all $ y \in Y $.
\end{theorem}
\begin{proof}
	First consider the case $s_X < \infty$, and fix $ r $.  Assume $y = K^Y g = LKg$ for some $g \in S$.  Thus,
	\begin{equation}\label{eqn:represent_y_Pi_rY_y}
	\Pi _r^Y y = \Pi _r^Y K^Y g = \Pi _r^Y LK g = \sum_{k=1}^{s_X} \sigma_k (g,f_k)_{S} \Pi _r^Y L \varphi_k,  \quad  \text{and}  \quad  y = \sum_{k=1}^{s_X} \sigma_k (g,f_k)_{S} L \varphi_k.
	\end{equation}
	%
	Subtracting gives
	$$\Pi _r^Y y - y = \sum_{k=1}^{s_X} \sigma_k (g,f_k)_{S} (\Pi _r^Y L \varphi_k - L \varphi_k)=\sum_{k=r+1}^{s_X} \sigma_k (g,f_k)_{S} (\Pi _r^Y L \varphi_k - L \varphi_k),$$
	since $\Pi_r^Y L \varphi_k = L \varphi_k$ for $ k = 1, ..., r$. 
	The error bound \eqref{eqn:PirY_error_bound} follows directly from this representation and the triangle inequality.  Furthermore, since $ s_X < \infty $, clearly $ \Pi_r^Y y \to y $ as $ r $ increases for each $ y \in \mathcal{R}(K^Y) $.
	
	Next, assume the POD eigenvalues for the data $ \{ L w_j \} $ are all nonzero.  By \Cref{prob4_part2} of \Cref{prob4}, since $ s_Y \leq s_X < \infty $ we have $ Y = \mathcal{R}(K^Y) $.  This gives $ \Pi _r^Y y \to y $ for all $ y \in Y $.
	
	
	Now consider the case $s_X = \infty$, and fix $ r $.  For $y = K^Y g = LKg$ with $ g \in S $ as above, recall the definition of $y_N = L \Pi_N^X Kg$ given in \eqref{yN}.  We have
	\begin{align*}
	\|\Pi_r^Y y - y \|_Y & \leq \|\Pi_r^Y y - \Pi_r^Y y_N \|_Y + \|\Pi_r^Y y_N -y_N\|_Y + \|y_N - y \|_Y \\
	& \leq \left( \| \Pi_r^Y \| + 1 \right) \| y - y_N \|_Y + \|\Pi_r^Y y_N -y_N\|_Y.
	\end{align*}
	Note that for the second term, $\|\Pi_r^Y y_N -y_N\|_Y$, we can obtain representations for $ \Pi_r^Y y_N $ and $ y_N $ similar to that in \eqref{eqn:represent_y_Pi_rY_y} above.  Proceeding in the same way gives		
	$$\|\Pi_r^Y y_N -y_N\|_Y \leq \sum_{k=r+1}^{N} \sigma_k |(g,f_k)_{S}| \, \|\Pi _r^Y L \varphi_k - L \varphi_k\|_Y.$$
	Since $r$ is fixed and $ y_N \to y $ as $ N \to \infty $ (\Cref{lemma:yN_simple_lemma}), the two inequalities above give
	$$ \|\Pi_r^Y y - y \|_Y \leq \sum_{k=r+1}^{\infty} \sigma_k |(g,f_k)_{S}| \, \|\Pi _r^Y L \varphi_k - L \varphi_k\|_Y.$$
	For convergence, we have
	$$
	\|\Pi _r^Y y - y\|_Y  \leq  \left( \sum_{k>r} |(g,f_k)_{S}|^2  \right)^{1/2} \, \left( \sum_{k>r} \sigma_k^2 \|\Pi _r^Y L \varphi_k - L \varphi_k\|_Y^2  \right)^{1/2}.
	$$
	Since $ \{ f_k \} $ is an orthonormal basis for $ S $, we know $ \sum_{k>r} |(g,f_k)_{S}|^2 $ goes to zero as $r$ increases by Parseval's equality. Furthermore, since $ \{ \Pi_r^Y \} $ is uniformly bounded, \Cref{HSconv} gives that $\sum_{k>r} \sigma_k^2 \|\Pi _r^Y L \varphi_k - L \varphi_k\|_Y^2 $ goes to zero as $r$ increases.  This gives $ \Pi_r^Y y \to y $ for each $ y \in \mathcal{R}(K^Y) $. 
	
	Finally, assume the POD eigenvalues for the data $ \{ L w_j \} $ are all nonzero.  By \Cref{prob4_part2} of \Cref{prob4}, we have $ \mathcal{R}(K^Y) $ is dense in $ Y $.  Since $ \{ \Pi_r^Y \} $ is uniformly bounded, the principle of uniform boundedness gives $ \Pi _r^Y y \to y $ for all $ y \in Y $.
\end{proof}

For the next two results we need to assume $ L $ or $ L^{-1} $ is bounded whenever $ s_X = \infty $.
%
%
\begin{theorem}\label{weirdprojX}
	Assume $s_X < \infty$, or either $ L $ or $ L^{-1} $ is bounded. If $y = K^Y g$ for some $g \in S$, then
	\begin{equation}\label{eqn:weirdprojX_error_bound}
	\|y - L \Pi_r^X L^{-1} y \|_Y \leq \sum_{k > r} \sigma_k\, |(g,f_k)_{S}| \, \|L\varphi_k\|_Y
	\end{equation}
	and the error converges to zero as $r$ increases.  Now assume $ \{L \Pi_r^X L^{-1}\}$ is uniformly bounded in operator norm whenever $ s_X = \infty $.  If the POD eigenvalues for the data $ \{ L w_j \} $ are all nonzero, then $ L \Pi_r^X L^{-1} y \to y $ for all $ y \in Y $.	 
\end{theorem}
%
%
\begin{proof}
	Let $ y = LKg$ for some $g \in S$, assume $s_X < \infty$, and fix $ r $.  As in the proof of \Cref{contyconvergence}, it can be shown that
	\begin{align*}
	y - L \Pi_r^X L^{-1} y &= L K g - L \Pi_r^X K g \\
	&= \sum_{k=1}^{s_X} \sigma_k (g,f_k)_{S} L \varphi_k - \sum_{k=1}^{r} \sigma_k (g,f_k)_{S} L \varphi_k \\
	&= \sum_{k=r+1}^{s_X} \sigma_k (g,f_k)_{S} L \varphi_k.
	\end{align*}
	The triangle inequality gives the error bound \eqref{eqn:weirdprojX_error_bound}.  The convergence results for the case $ s_X < \infty $ follow just as in the proof of \Cref{contyconvergence}.
	
	
	Now consider the case $s_X = \infty$, assume $ y = LKg$ for some $g \in S$, and fix $ r $. Then for $y_N = L \Pi_N^X Kg$ as in \eqref{yN}, with $ N \geq r $, we have
	$$ \|y - L \Pi_r^X L^{-1} y \|_Y \leq \|y - y_N\|_Y + \|y_N - L \Pi_r^X L^{-1} y_N \|_Y + \| L \Pi_r^X L^{-1} y_N - L \Pi_r^X L^{-1} y \|_Y.$$
	\Cref{lemma:yN_simple_lemma} implies that the first term tends to zero as $ N \to \infty $.  For the second term, proceed as above and use $ \Pi_r^X \Pi_N^X = \Pi_r^X $ (since $ N \geq r $) to show
	$$\| y_N - L \Pi_r^X L^{-1} y_N \|_Y \leq \sum_{k=r+1}^{N} \sigma_k\, |(g,f_k)_{S}| \, \|L\varphi_k\|_Y.$$
	For the third term, first assume $ L $ is bounded.  In this case,
	$$
	\| L \Pi_r^X L^{-1} y_N - L \Pi_r^X L^{-1} y \|_Y  \leq  \| L \| \| \Pi_r^X \| \| L^{-1} y_N - L^{-1} y \|_X = \| L \| \| \Pi_r^X \| \| \Pi_N^X K g - K g \|_X,
	$$
	which converges to zero as $ N \to \infty $, since $ r $ is fixed.  If instead $ L^{-1} $ is bounded, then $ L \Pi_r^X L^{-1} $ is bounded by \Cref{boundedext} and so
	$$
	\| L \Pi_r^X L^{-1} y_N - L \Pi_r^X L^{-1} y \|_Y  \leq  \| L \Pi_r^X L^{-1} \| \| y_N - y \|_Y,
	$$
	which converges to zero as $ N \to \infty $ by \Cref{lemma:yN_simple_lemma}, again since $ r $ is fixed.  Combining the above results gives
	$$
	\|y - L \Pi_r^X L^{-1} y \|_Y  \leq  \sum_{k=r+1}^\infty \sigma_k\, |(g,f_k)_{S}| \, \|L\varphi_k\|_Y.
	$$
	
	For convergence, we proceed as in the proof of \Cref{contyconvergence}.  We have
	$$
	\|y - L \Pi_r^X L^{-1} y \|_Y  \leq  \left( \sum_{k>r} |(g,f_k)_{S}|^2  \right)^{1/2} \, \left( \sum_{k>r} \sigma_k^2 \| L \varphi_k\|_Y^2  \right)^{1/2}.
	$$
	We know $ \sum_{k>r} |(g,f_k)_{S}|^2 $ goes to zero as $r$ increases by Parseval's equality. Furthermore, \Cref{HSconv} gives that $\sum_{k>r} \sigma_k^2 \|L \varphi_k\|_Y^2 $ goes to zero as $r$ increases.  This implies $ L \Pi_r^X L^{-1} y \to y $ for each $ y \in \mathcal{R}(K^Y) $.  To show convergence for all $ y \in Y $, we again use \Cref{prob4_part2} of \Cref{prob4} and the principle of uniform boundedness.
	%
	%
\end{proof}



We omit the proof of the next result, as it is similar to the proof of the previous result, \Cref{weirdprojX}.  Note that in \Cref{weirdprojX} the error converges to zero for a fixed $ y \in \mathcal{R}(K^Y) $ without any additional assumptions.  In this next result, if $ s_X = \infty $ we need to require additional conditions to guarantee that the error converges to zero for a fixed $ x \in \mathcal{R}(K) $; these conditions come from \Cref{HSconv}.
\begin{theorem}\label{weirdprojpsi}
	Assume $s_X < \infty$ or either $L$ or $L^{-1}$ is bounded. If $x = Kg$ for some $g \in S$, then  
	\begin{equation}\label{eqn:weirdprojpsi_error_bound}
	\|x- L^{-1} \Pi_r^Y Lx\|_X \leq \sum_{k>r} \sigma_k  |(g,f_k)_{S}| \,  \|\varphi_k - L^{-1} \Pi_r^Y L \varphi_k\|_X.
	\end{equation}
	If $s_X <\infty$, the error converges to zero as $r$ increases. If $s_X = \infty$, then the error goes to zero as $r$ increases when either (i) $L^{-1}$ is bounded and $\{ \Pi_r^Y \}$ is uniformly bounded or (ii) $\{L^{-1}\Pi_r^Y L\}$ is uniformly bounded.  Now assume $ \{L^{-1} \Pi_r^Y L\}$ is uniformly bounded in operator norm whenever $ s_X = \infty $.  If the POD eigenvalues for the data $\{w_j\}$ are all nonzero, then $L^{-1} \Pi_r^Y L x \to x$ for all $x \in X$.
\end{theorem}

To be complete, we give an exact error formula and an error bound for approximations of elements in the range of $ K $ using the POD projection $ \Pi_r^X $.  This result gives an error bound for approximating each individual data snapshot in the discrete case.
\begin{theorem}\label{error_bound_rangeK}
	If $ x = K g $ for some $ g \in S $, then
	\begin{equation}\label{eqn:error_bound_rangeK_general}
	  \| x - \Pi_r^X x \|_X  =  \left(  \sum_{k>r}  \sigma_k^2 | (g,f_k)_S |^2 \right)^{1/2}  \leq  \sigma_{r+1} \| g \|_S.
	\end{equation}
	Also, in the discrete case, for each $ \ell = 1, \ldots, s $ we have
	\begin{equation}\label{eqn:error_bound_rangeK_snapshot}
	  \| w_\ell - \Pi_r^X w_\ell \|_X  \leq  \gamma_\ell^{-1/2} \sigma_{r+1}.
	\end{equation}
\end{theorem}
%

\begin{remark}
	The bound \eqref{eqn:error_bound_rangeK_snapshot} was obtained in \cite[Proposition 3.1]{Kostova-VassilevskaOxberry18} for $ X = \mathbb{R}^n $ and $ \gamma_\ell = 1 $ for all $ \ell $.  Recall the constants $ \{ \gamma_\ell \} $ are the positive weights in the definition of the POD operator $ K $ in the discrete case; see \Cref{discretebackground}.
\end{remark}
\begin{proof}
  	Using the SVD of $ K $ gives
  	$$
  	  x - \Pi_r^X x = \sum_{k>r} \sigma_k (g,f_k)_S \varphi_k.
  	$$
  	Since $ \| x - \Pi_r^X x \|_X^2 = ( x - \Pi_r^X x, x - \Pi_r^X x )_X $ and $ \{ \varphi_k \} $ is an orthonormal basis for $ X $, we immediately obtain the exact error formula in \eqref{eqn:error_bound_rangeK_general}.  To obtain the error bound in \eqref{eqn:error_bound_rangeK_general}, use $ \sigma_k \leq \sigma_{r+1} $ for all $ k > r $ and also Parseval's equality.
  	
  	Next, in the discrete case we have $ w_\ell = K g_\ell $ for each $ \ell = 1, \ldots, s $, where $ g_\ell = \gamma_\ell^{-1} e_\ell $ and $ e_\ell $ is the $ \ell $th standard unit vector for $ \mathbb{K}^s $, i.e., the $ \ell $th entry of $ e_\ell $ is one and all other entries are zero.  The error bound \eqref{eqn:error_bound_rangeK_snapshot} follows from $ \| g_\ell \|_S = \gamma_\ell^{-1/2} $ and \eqref{eqn:error_bound_rangeK_general}.
\end{proof}

In \Cref{error_bound_rangeK}, note that the quantity $ \gamma_\ell^{-1} $ appears in the error bound \eqref{eqn:error_bound_rangeK_snapshot} for approximating the snapshot $ w_\ell $.  However, in applications it is typical that each weight $ \gamma_\ell $ tends to zero as the number $ s $ of snapshots increases.  Next, we use the above results to prove various approximation error bounds for each individual snapshot $ w_\ell $ in the discrete case that do not depend on $ \gamma_\ell^{-1} $.  Here, the bounds are only valid if $ r $ is sufficiently large.  We note that these type of error bounds have been \textit{assumed} to hold in the literature; Iliescu and Wang made this type of assumption in \cite[Assumption 3.2]{IliescuWang14} (with $ \gamma_\ell = s^{-1} $ for all $ \ell $) in their analysis of a POD reduced order model of the Navier-Stokes equations, and many others have followed their approach.
\begin{corollary}\label{cor:error_bounds_each_snapshot}
  In the discrete case, if $ r $ is sufficiently large, then for each $ \ell = 1, \ldots, s $ we have
  \begin{subequations}
  \begin{align}
    \| w_\ell - \Pi_r^X w_\ell \|^2_X  &\leq  \sigma_{r+1}^2,\label{eqn:snapshot_error_bound1}\\
    \| L w_\ell - \Pi_r^Y L w_\ell \|^2_Y  &\leq  \sum_{k>r} \sigma_k^2 \| L \varphi_k - \Pi_r^Y L \varphi_k \|_Y^2,\label{eqn:snapshot_error_bound2}\\
	\| L w_\ell - L \Pi_r^X w_\ell \|^2_Y  &\leq  \sum_{k>r} \sigma_k^2 \| L \varphi_k \|_Y^2,\label{eqn:snapshot_error_bound3}\\
	\| w_\ell - L^{-1} \Pi_r^Y L w_\ell \|^2_X  &\leq  \sum_{k>r} \sigma_k^2 \| \varphi_k - L^{-1} \Pi_r^Y L \varphi_k \|_X^2.\label{eqn:snapshot_error_bound4}
  \end{align}
  \end{subequations}
\end{corollary}
\begin{proof}
  We only prove \eqref{eqn:snapshot_error_bound2}; the proofs of the remaining inequalities are similar.  As in the proof of \Cref{error_bound_rangeK}, we know $ w_\ell = K g_\ell $ for each $ \ell = 1, \ldots, s $, where $ g_\ell = \gamma_\ell^{-1} e_\ell $.  Using the error bound \eqref{eqn:PirY_error_bound} in \Cref{contyconvergence}, the Cauchy-Schwarz inequality on the sum, and Parseval's inequality gives
  $$
    \| L w_\ell - \Pi_r^Y L w_\ell \|^2_Y  \leq  \| g_\ell - \Pi_r^S g_\ell \|_S^2  \sum_{k>r} \sigma_k^2 \| L \varphi_k - \Pi_r^Y L \varphi_k \|_Y^2,
  $$
  where $ \Pi_r^S : S \to S $ is the orthogonal projection onto $ S_r := \mathrm{span}\{ f_k \}_{k=1}^r $.  Since $ \{ f_k \}_{k\geq 1} $ is an orthonormal basis for $ S $, we know $ \Pi_r^S g_\ell \to g_\ell $ for $ \ell = 1, \ldots, s $.  Since $ s $ is fixed, for all sufficiently large $ r $ we have $ \| g_\ell - \Pi_r^S g_\ell \|_S \leq 1 $ for all $ \ell = 1, \ldots, s $, and this completes the proof.
\end{proof}

\section{More Examples}\label{Section:Examples}

We now consider a few additional examples. For all three examples we consider two separable Hilbert spaces, $H$ and $V$, where $V$ is a proper subset of $H$, and $V$ is both continuously embedded\footnote{i.e., there exists a constant $ C_V > 0 $ such that $ \| v \|_H \leq C_V \| v \|_V $ for all $ v \in V $} and dense in $H$.  The linear operator $L$ is a mapping between these two spaces.

For all three examples, we present results for the continuous case only.  We assume we have the data $\{w_j\}_{j=1}^m \subset L^2(\mathcal{O};H) \cap L^2(\mathcal{O};V)$.  Results for the discrete case can also be obtained using the theory in this work if desired.

The first two examples are from our previous work, \cite{Singler14}.  Due to the above assumption on the data, the POD operator $ K $ can be viewed as a mapping into $ H $ or a mapping into $ V $.  One can obtain the SVD of $ K : S \to H $ or the SVD of $ K : S \to V $, i.e., one can choose $ X = H $ or $ X = V $.  The different choices for $ X $ give different POD singular values, POD singular vectors, POD modes, and POD projections.  In \cite{Singler14}, we considered both choices for $ X $ and four different POD projections between these spaces and gave exact expressions for the POD data approximation errors in the two different Hilbert space norms.  We relate the notation and results for both the error formulas and pointwise convergence from the present work to \cite{Singler14}.  We obtain better pointwise convergence results in this work.  Also, $\mathcal{O}$ was only an interval in \cite{Singler14}, but now we have $\mathcal{O} $ is an open subset of $ \mathbb{R}^d$.  For these first two examples, $Y_r = \mathrm{span}\{L\varphi_k\}$ and $\Pi_r^Y: Y \to Y $ is the orthogonal projection onto $Y_r$. Note this implies $ \{\Pi_r^Y\} $ is uniformly bounded in operator norm. 

For the third example, we consider a case where $\Pi_r^Y$ is not an orthogonal projection.  In particular, we take $\Pi_r^Y$ to be a Ritz projection, as considered in \cite{IliescuWang13,Rubino18}.  All of our results for this case are new.


\subsection{Example 1}\label{example1}

For the first example, consider the case where $X = H$, $Y = V$, and $L:H \to V$ is defined by $Lv =v $ for all $v \in \mathcal{D}(L) =V$.  The operator $ L $ is clearly invertible, and $ L^{-1} : V \to H $ is given by $ L^{-1} v = v $ for all $ v \in V $.  Note that $L^{-1}: V \to H$ is bounded due to the continuous embedding assumption. Also, the inverse of a bounded operator is closed, so $L$ is closed.  Furthermore, the assumption on the data gives $ \{ w_j \} \subset L^2(\mathcal{O};X) $ and $ \{ Lw_j \} \subset L^2(\mathcal{O};Y) $. Thus, we know that both the main assumption and the $L^{-1}$ assumption hold.

Since $ X = H $ and each set of singular vectors of the POD operator $ K : S \to H $ are an orthonormal basis, we know the POD modes $ \{ \varphi_k \} $ are an orthonormal basis for $ H $.  Note that $ X_r = \mathrm{span}\{ \varphi_k \}_{k=1}^r \subset H $, and $ Y_r = \mathrm{span}\{ L \varphi_k \}_{k=1}^r = \mathrm{span}\{ \varphi_k \}_{k=1}^r \subset V $.  Furthermore, the POD modes $ \{ \varphi_k \} $ may not be orthogonal in $ V $.  Also, the operator $K^Y = L K $ is simply the POD operator $ K $ viewed as a mapping from $ S $ to $ V $.  We take $ \Pi_r^X :X \to X $ to be the orthogonal projection onto $ X_r $, and $ \Pi_r^Y :Y \to Y $ to be the orthogonal projection onto $ Y_r $.

In order to discuss the POD projections we pay special attention to the spaces under consideration. Since $V \subset H$, the projections can be considered as mappings from $V$ to $V$ or from $H$ to $H$. The projections considered in this work are related to the projections $ P_r^H $ and $ P_r^V $ in \cite[Definition 3.2]{Singler14} as follows:	
\begin{itemize}
	\item $\Pi_r^X: X \to X$ is equal to the orthogonal projection $P_r^H: H \to H$. 
	\item $\Pi_r^Y: Y \to Y$ is equal to the orthogonal projection $P_r^V: V \to V$. 
	\item $L \Pi_r^X L^{-1}: Y \to Y$ is equal to the operator $P_r^H: V \to V$.
	\item $L^{-1} \Pi_r^Y L: X \to X$ is equal to the operator $P_r^V: H \to H$.
\end{itemize} 	

Now that we have the relationships between the projections, we compare the results. The error formulas presented here in \Cref{conterror} are essentially the same as the results in \cite{Singler14}.  Again, the primary difference here is that $\mathcal{O} $ is an open subset of $ \mathbb{R}^d$ instead of an interval.  The POD data approximation errors from \Cref{conterror} become the following:
\begin{align}
\sum_{j=1}^{m} \int_\mathcal{O} \| w_j(t) - P _r^H w_j(t) \|^2_{V} dt  &=  \sum_{k>r} \sigma_k^2 \| \varphi_k \|^2_V,\label{eqn:ex1_error_formula_1}\\
\sum_{j=1}^{m} \int_\mathcal{O} \| w_j(t) -  P_r^V w_j(t) \|^2_{V} dt  &= \sum_{k>r} \sigma_k^2 \| \varphi_k - P_r^V \varphi_k \|^2_V,\label{eqn:ex1_error_formula_2}\\
\sum_{j=1}^{m} \int_\mathcal{O} \| w_j(t) - P_r^V w_j(t) \|^2_{H} dt  &= \sum_{k>r} \sigma_k^2 \| \varphi_k -P_r^V \varphi_k \|^2_H \label{eqn:ex1_error_formula_3}.
\end{align}
In this example, all three sums converge to zero as $r$ increases. 

A larger improvement from \cite{Singler14} can be seen in the results concerning pointwise convergence of POD projections.  To illustrate, we give the following result.
\begin{proposition}\label{translate convergence example 1}
	We have
	\begin{enumerate}
		\item $\|P_r^V y -y\|_V \to 0$ for all $y \in \mathcal{R}(K)$, and for $ y = K g $ we have
		$$
		\|P_r^V y -y\|_V \leq \sum_{k>r} \sigma_k |(g,f_k)_{S}| \, \|P_r^V \varphi_k - \varphi_k\|_V.
		$$
		\item If the POD eigenvalues for $\{w_j\} \subset L^2(\mathcal{O};V)$ are all nonzero, then $ P_r^V y \to y $ in both $ H $ and $ V $ for all $y \in V$.
		\item $\|P_r^H y - y \|_V \to 0$ for all $y \in \mathcal{R}(K)$, and for $ y = K g $ we have
		$$
		\|y - P_r^H y \|_V \leq \sum_{k > r} \sigma_k\, |(g,f_k)_{S}| \, \|\varphi_k\|_V.
		$$
		\item $\|P_r^V x - x \|_H \to 0$ for all $x \in \mathcal{R}(K)$, and for $ x = K g $ we have
		$$
		\|x - P_r^V x \|_H \leq \sum_{k > r} \sigma_k\, |(g,f_k)_{S}| \, \|\varphi_k - P_r^V \varphi_k\|_H.
		$$
	\end{enumerate}
\end{proposition}
Note that since $\Pi_r^Y$ is orthogonal, item 1 and item 2 follow from \Cref{contyconvergence} and item 2 of \Cref{convergencewithboundedop}. Items 3 and 4 can be obtained from \Cref{weirdprojX}, \Cref{weirdprojpsi}, and the fact that $L^{-1}$ is bounded.

The pointwise convergence results above are more complete and more sharp than the results in \cite[Proposition 5.5]{Singler14}.  First, item 2 is shown in \cite[Proposition 5.5]{Singler14} under the assumption that all the POD singular values for $\{w_j\} \subset L^2(\mathcal{O};V)$ are nonzero; as discussed in \Cref{Section:basic} this is a more restrictive assumption than the POD eigenvalues all being nonzero, as is required above.  Next, the convergence result in item 3 is shown in \cite[Proposition 5.5]{Singler14}; however, the error bound in item 3 is new.  Also, items 1 and 4 are completely new.

For item 3, we note that an error bound was given in the proof of \cite[Proposition 5.5]{Singler14}. However, that error bound does not converge to zero as fast as the error bound given in \Cref{weirdprojX}.  Specifically, the error bound in \cite{Singler14} is a constant multiple of $ ( \sum_{k > r} |(g,f_k)_{S}|^2 )^{1/2} $.  However, the error bound in item 3 can be bounded above by
$$
\|y - P_r^H y \|_V \leq \bigg( \sum_{k > r}  |(g,f_k)_{S}|^2  \bigg)^{1/2}  \bigg( \sum_{k > r}  \sigma_k^2\, \|\varphi_k\|_V^2 \bigg)^{1/2},
$$
and both terms in parentheses tend to zero as $ r $ increases by Parseval's equality and \Cref{HSconv} (see the proof of \Cref{weirdprojX}).  Therefore, the error bound in item 3 is an improvement over the error bound in \cite{Singler14}.

Finally, we consider boundedness of the non-orthogonal POD projections $P_r^H: V \to V$ and $P_r^V: H \to H$.  For each fixed $ r $, we showed in \cite[Lemma 3.3]{Singler14} that $P_r^H: V \to V$ is bounded.  We did not consider the boundedness of $P_r^V: H \to H$ in \cite{Singler14}.  Below, we use \Cref{boundedext} to show $P_r^H: V \to V$ is bounded and also give a condition guaranteeing $P_r^V: H \to H$ has a bounded extension.  However, we still do not know if these non-orthogonal POD projections are uniformly bounded in operator norm. 


Define the linear operator $A: \mathcal{D}(A) \subset H \to H$ by $$(Au,v)_H = (u,v)_V$$ for all $u \in \mathcal{D}(A)$ and $v \in V$ (see, e.g., \cite[Section II.2]{Temam97}). We know $A$ is closed. Now we apply this to our example. For all $x \in \mathcal{D}(L) = V$ and $y \in \mathcal{D}(L^*)$ we have
\begin{align*}
(x,L^*y)_H &= (Lx,y)_V = (x,y)_V \\
\Rightarrow  (L^* y, x)_H &= (y,x)_V.
\end{align*}
Thus, $L^* = A$ and $\mathcal{D}(L^*) = \mathcal{D}(A)$. For PDE solution data we often have $ \{Aw_j\} \subset L^2(\mathcal{O};H)$ for each $j$; see \cite{Temam97} for examples.  In this case, since $ \varphi_k = \sigma_k^{-1} K f_k $ we can use the Bochner integral result in \Cref{bochnerint} to show $ \varphi_k \in \mathcal{D}(A) $ whenever $ \sigma_k > 0 $.

Therefore, since $ L^{-1} $ is bounded, item 1 and item 4 of \Cref{boundedext} give the following result.
%
\begin{proposition}
	Let $ r $ be fixed.  The operator $P_r^H: V \to V$ is bounded, and if $\{Aw_j\}_{j=1}^m \subset L^2(\mathcal{O};H)$, then the operator $P_r^V: H \to H$ can be extended to a bounded operator.
\end{proposition}

\subsection{Example 2}\label{example2}

Next, consider the case where $X = V$, $Y = H$, and $L:V \to H$ is defined by by $Lv =v $ for all $v \in V$.  Then $L^{-1}: H \to V$ is given by $L^{-1}v = v$ for all $ v \in \mathcal{D}(L^{-1})=V $.  Note that in this case $L$ is bounded by the continuous embedding property.  Again, the assumption on the data gives $ \{ w_j \} \subset L^2(\mathcal{O};X) $ and $ \{ Lw_j \} \subset L^2(\mathcal{O};Y) $.  Therefore, the main assumption and the $ L^{-1} $ assumption hold.  

Since $ X = V $, in this example the POD modes $ \{ \varphi_k \} $ are an orthonormal basis for $ V $.  We have $ X_r = \mathrm{span}\{ \varphi_k \}_{k=1}^r \subset V $, and $ Y_r = \mathrm{span}\{ L \varphi_k \}_{k=1}^r = \mathrm{span}\{ \varphi_k \}_{k=1}^r \subset H $.  The POD modes $ \{ \varphi_k \} $ may not be orthogonal in $ H $.  The operator $K^Y = L K $ is the POD operator $ K : S \to H $.  As in Example 1, $ \Pi_r^X :X \to X $ is the orthogonal projection onto $ X_r $, and $ \Pi_r^Y :Y \to Y $ is the orthogonal projection onto $ Y_r $.

The projections in this work are related to the projections $ Q_r^H $ and $ Q_r^V $ from \cite[Definition 3.2]{Singler14} as follows:
\begin{itemize}
	\item $\Pi_r^X:X \to X$ is equal to the orthogonal projection $ Q_r^V: V \to V$.
	\item $\Pi_r^Y: Y \to Y$ is equal to the orthogonal projection $Q_r^H: H \to H$.
	\item $L \Pi_r^X L^{-1}: Y \to Y$ is equal to the operator $Q_r^V: H \to H$.
	\item $L^{-1} \Pi_r^Y L: X \to X$ is equal to the operator $Q_r^H: V \to V$. 
\end{itemize}

As before, the main data approximation error results in \Cref{conterror} become
\begin{align*}
\sum_{j=1}^{m} \int_\mathcal{O} \| w_j(t) - Q_r^V w_j(t) \|^2_{H} dt  &=  \sum_{k>r} \sigma_k^2 \| \varphi_k \|^2_H,\\
\sum_{j=1}^{m} \int_\mathcal{O} \| w_j(t) -  Q_r^H w_j(t) \|^2_{H} dt  &= \sum_{k>r} \sigma_k^2 \|\varphi_k - Q_r^H \varphi_k \|^2_H,\\
\sum_{j=1}^{m} \int_\mathcal{O} \| w_j(t) - Q_r^H  w_j(t) \|^2_{V} dt  &= \sum_{k>r} \sigma_k^2 \| \varphi_k - Q_r^H \varphi_k \|^2_V.
\end{align*}
Here the first two sums converge to zero as $r$ increases. However, we cannot show convergence of the last sum. This is because we do not know $L^{-1}$ is bounded or $\{Q_r^H\}$ is uniformly bounded as a family of operators mapping $ V $ to $ V $.  As before, the only improvement here compared to \cite{Singler14} is that $ \mathcal{O} $ is not restricted to be an interval.

We also have the following pointwise convergence results. 
\begin{proposition} As $r$ increases we have
	\begin{enumerate}
		\item $\|Q_r^H y - y \|_H \to 0$ for all $y \in H$, and for $ y = K g $ we have
		$$
		\|Q_r^H y - y\|_H \leq \sum_{k>r} \sigma_k |(g,f_k)_{S}| \, \|Q_r^H \varphi_k - \varphi_k\|_H.
		$$
		\item $\|Q_r^V y - y \|_H \to 0$ for all $y \in V$, and for $ y = K g $ we have
		$$
		\|Q_r^V y - y\|_H \leq \sum_{k>r} \sigma_k |(g,f_k)_{S}| \, \|\varphi_k\|_H.
		$$
		\item For $ x = K g $ we have
		$$
		\|Q_r^H x - x\|_V \leq \sum_{k > r} \sigma_{k} | (g,f_k)_S | \|Q_r^H \varphi_k - \varphi_k \|_V.
		$$
		If also $s_X < \infty$ or $L^{-1}$ is bounded, then the error goes to zero as $r$ increases.
	\end{enumerate}
\end{proposition}
Since $L$ is bounded, item 1 follows from item 1 of \Cref{bounded convergence of Pi_r^Psi} and also \Cref{contyconvergence}.  Item 2 can be obtained from \Cref{weirdprojX}, using $ L $ is bounded.  \Cref{weirdprojpsi} gives item 3; note that we cannot guarantee convergence of the error without the extra assumptions since we only know $ L $ is bounded.

Again, these results improve on the results in \cite[Proposition 5.5]{Singler14}.  All of the error bounds are new.  The convergence result in item 2 was not stated in \cite{Singler14}, but it follows directly from the continuous embedding and $ \| Q_r^V y - y \|_V \to 0 $ for all $ y \in V $.  The convergence result in item 1 was given in \cite[Proposition 5.5]{Singler14}, however we made the assumption that all the POD singular values for $\{w_j\} \subset L^2(\mathcal{O};V)$ are nonzero.  Here, we proved the convergence result in item 1 without that assumption.

Next, we use the technique from \Cref{example1} to determine the boundedness of the non-orthogonal POD projections $ Q_r^H : V \to V $ and $ Q_r^V : H \to H $.  For this example, we have $A = L^{-*} = (L^{-1})^*$.  Therefore, if $\{A w_j\} \subset L^2(\mathcal{O};H)$, then we have $\{ \varphi_k \} \subset \mathcal{D}(L^{-*})$, just as in \Cref{example1}.  Since $ L $ is bounded, items 2 and 3 of \Cref{boundedext} give the following result.
%
%
\begin{proposition}
	Let $ r $ be fixed.  The operator $Q_r^H: V \to V$ is bounded, and if $\{Aw_j\}_{j=1}^m \subset L^2(\mathcal{O};H)$, then the operator $Q_r^V: H \to H$ can be extended to a bounded operator on $H$.
\end{proposition}

\subsection{Example 3}\label{example4}

In order to demonstrate the usefulness of considering $\Pi_r^Y$ as a non-orthogonal projection, we consider the case of a Ritz projection as presented in \cite{IliescuWang13,Rubino18}. 

Consider the situation from Example 1 in \Cref{example1}: we have $X = H$, $Y = V$, and $L: X \to Y$ is defined by $Lv = v$ for all $v \in \mathcal{D}(L) = Y$.  Assume we have a continuous elliptic sesquilinear form\footnote{i.e., there exists constants $ C_a, c_a > 0 $ such that $ | a(u,v) | \leq C_a \| u \|_V \| v \|_V $ and $ c_a \| u \|_V^2 \leq \mathrm{Re} \, a(u,u) $ for all $ u, v \in V $} $a: V \times V \to \mathbb{K}$. Define the projection $P_r^V : V \to V$ onto $V_r := Y_r = \text{span}\{L\varphi_k\} = \text{span}\{\varphi_k\} \subset V$ as follows: let $P_r^V u := u_r \in V_r$ be the unique solution of 
$$
a(u_r, v_r) = a(u, v_r)  \quad  \mbox{for all $v_r \in V_r$.}
$$
The existence and uniqueness of such a solution is guaranteed by the Lax-Milgram Theorem.  We take $\Pi_r^Y = P_r^V$.

Note that the main difference between this example and Example 1 is that the projection $P_r^V$ is not the same.  However, for this example it can be checked that the family of projections, $\{ \Pi_r^Y \}$, is uniformly bounded.  Therefore, the same pointwise convergence results and error formulas from \Cref{example1} hold for this example with $ P_r^V : V \to V $ defined as above.  We note that these pointwise convergence results and error formulas are all new.  Bounds on the POD data approximation errors can be found in Lemma 3.4 in \cite{IliescuWang13} and Lemma 2.9 in \cite{Rubino18} in the discrete case; however, we have the exact formulas \eqref{eqn:ex1_error_formula_1}-\eqref{eqn:ex1_error_formula_3} for the POD data approximation errors in the continuous case.  Again, analogous error formulas can be derived for the discrete case using our results.

\section{Conclusions}

We proved new generalized error formulas for POD data approximation errors for both the discrete and continuous cases. We also showed convergence of these errors under certain conditions, and obtained new pointwise convergence results for POD projections. We demonstrated the application of our results to several example problems.  We leave the application of these results to the numerical analysis of POD model order reduction methods for PDEs to be considered elsewhere.



Some open questions remain.  When $L^{-1}$ is unbounded, we had to assume uniform boundedness of the POD projections $\{L^{-1} \Pi_r^Y L\}$ to show that the error formula in \eqref{last cont error} converges to zero as $r$ increases. We do not know if there is a simpler condition that yields convergence of the approximation error.  If $ L $ or $ L^{-1} $ is unbounded, we also do not know if the POD projections $\{ L \Pi_r^X L^{-1} \}$ and $\{ L^{-1} \Pi_r^Y L \}$ are uniformly bounded.  Both of these issues have been discussed in the context of Example 2 in \Cref{example2} in \cite{Chapelle12,Singler14}.  The second issue has also been discussed in the context of Example 1 in \Cref{example1} in \cite{XieWellsWangIliescu18,KeanSchneier19}; in these works, the $H^1$ stability of the $L^2$ POD projection is of interest.


\appendix
\section{Optimality of Discrete and Continuous POD}
\label{sec:optimality_POD_proof}

To be complete, we present a brief proof of the optimality of POD for low rank data approximation in both the discrete and continuous cases.  Our problem statement and proof strongly rely on ideas from \cite{Volkwein04} and \cite{Djouadi08}.

\textbf{POD optimality problem:}  Let $ X $ be a separable Hilbert space, and let $ S = \mathbb{K}^s_\Gamma $ in the discrete case or $ S = L^2(\mathcal{O};\mathbb{K}^m) $ in the continuous case, where $ \mathbb{K} = \mathbb{R} $ or $ \mathbb{K} = \mathbb{C} $; see \Cref{Section:basic} for details.  Suppose we have given data $ \{ w_j \}_{j=1}^s \subset X $ in the discrete case or $ \{ w_j \}_{j=1}^m \subset L^2(\mathcal{O};X) $ in the continuous case.  The POD optimality problem is to find coefficients $ \{ a_{k} \} \subset \mathbb{K} $ and basis elements $ \{ s_k \} \subset S $ and $ \{ \eta_k \} \subset X $ so that the $ r $th order approximations
\begin{align*}
  w_j^r &= \sum_{k=1}^r a_{k} s_{k,j} \eta_k  &  &\mbox{for $ j = 1, \ldots, s $ (discrete case),}\\
  w_j^r(t) &= \sum_{k=1}^r a_{k} s_{k,j}(t) \eta_k  &  &\mbox{for $ j = 1, \ldots, m $ (continuous case),}
\end{align*}
minimize the data approximation error
\begin{align*}
E_r(w^r) &= \sum_{j=1}^s \gamma_j \| w_j - w_j^r \|_X^2  &  &\mbox{(discrete case),}\\
E_r(w^r) &= \sum_{j=1}^m \int_{\mathcal{O}} \| w_j(t) - w_j^r(t) \|_X^2 \, dt  &  &\mbox{(continuous case).}
\end{align*}
\begin{remark}
	In many papers on POD, the basis elements $ \{ \eta_k \} \subset X $ are required to be orthonormal, and $ w_j^r $ is also required to equal the orthogonal projection of $ w_j $ onto $ \mathrm{span}\{ \eta_k \}_{k=1}^r $.  Therefore, the POD problem above allows more general approximations.  The final result is the same.
\end{remark}
%

\textbf{Notation:}  For given data $ \{ y_j \}_{j=1}^s \subset X $ in the discrete case or $ \{ y_j \}_{j=1}^m \subset L^2(\mathcal{O};X) $ in the continuous case, we let $ K(y) : S \to X $ denote the POD operator for the data and we let $ K^*(y) : X \to S $ denote the Hilbert adjoint operator of $ K(y) $.

The proof of the next result follows directly from definitions and is omitted.
\begin{lemma}\label{prop:POD_op_finite_dim_data}
	If the data is given by
	$$
	  y_j = \sum_{k=1}^p \alpha_k s_{k,j} \eta_k  \quad  \mbox{(discrete case)},  \quad  y_j(t) = \sum_{k=1}^p \alpha_{k} s_{k,j}(t) \eta_k,  \quad  \mbox{(continuous case)}
	$$
	for each $ j $ with $ \{ \alpha_{k} \} \subset \mathbb{K} $, $ \{ s_k \} \subset S $, and $ \{ \eta_k \} \subset X $, then the POD operator $ K(y) : S \to X $ is given by
	$$
	K(y) f = \sum_{k=1}^p \alpha_{k} (f,\overline{s_{k}})_S \eta_k,  \quad  f \in S.
	$$
\end{lemma}

Next, we present the discrete version of the Hilbert-Schmidt norm result for a continuous POD operator in \Cref{HS}.
\begin{lemma}\label{prop:HS_discrete}
	For given data $ \{ y_j \}_{j=1}^s \subset X $ in the discrete case, the Hilbert-Schmidt norm of the POD operator $ K(y) : S \to X $ is given by
	$$
	  \| K(y) \|_{\mathrm{HS}(S,X)}^2 = \sum_{j=1}^s \gamma_j \| y_j \|_X^2.
	$$
\end{lemma}
\begin{proof}
  Let $ \{ \xi_k \}_{k \geq 1} $ be an orthonormal basis for $ X $.  We have
  \begin{align*}
    \| K(y) \|_{\mathrm{HS}(S,X)}^2  &=  \| K^*(y) \|_{\mathrm{HS}(X,S)}^2  =  \sum_{k \geq 1} \| K^*(y) \xi_k \|_S^2  =  \sum_{k \geq 1}  \sum_{j=1}^s  \gamma_j | ( \xi_k, y_j )_X |^2  \\
      &=  \sum_{j=1}^s  \gamma_j  \sum_{k \geq 1} | ( \xi_k, y_j )_X |^2  =  \sum_{j=1}^s  \gamma_j  \| y_j \|_X^2
  \end{align*}
  by Parseval's inequality.
\end{proof}

Now we prove the main optimality result.  We rely on the fact that the rank $ r $ truncated SVD of $ K(y) $ is the optimal rank $ r $ approximation to $ K(y) $ in the Hilbert-Schmidt norm; see, e.g., \cite[Section III.7, Theorem 7.1]{GohbergKrein69}.
\begin{theorem}
	Let $ \{ w_j \}_{j=1}^s \subset X $ in the discrete case or $ \{ w_j \}_{j=1}^m \subset L^2(\mathcal{O};X) $ in the continuous case be given data, and let $ \{ \sigma_i, f_i, \varphi_i \} \subset \mathbb{R} \times S \times X $ be the ordered singular values of $ K(w) : S \to X $ and the corresponding orthonormal bases of singular vectors.  A solution of the POD problem is given by $ \{ w_j^r \}_{j=1}^s \subset X $ in the discrete case or $ \{ w_j^r \}_{j=1}^m \subset L^2(\mathcal{O};X) $ in the continuous case, where
	\begin{align*}
	  w_j^r &= \sum_{k = 1}^r \sigma_k \overline{f_{k,j}} \varphi_k = \sum_{k = 1}^r \big( w_j, \varphi_k \big)_X \varphi_k  &  &\mbox{(discrete case),}\\
	  w_j^r(t) &= \sum_{k = 1}^r \sigma_k \overline{f_{k,j}(t)} \varphi_k = \sum_{k = 1}^r \big( w_j(t), \varphi_k \big)_X \varphi_k  &  &\mbox{(continuous case).}
	\end{align*}
	The minimum approximation error is given by
	$$
	E_r^\mathrm{min} := E_r(w^r) = \sum_{k > r} \sigma_k^2 < \infty,
	$$
	and $ E_r^\mathrm{min} \to 0 $ as $ r $ increases.
\end{theorem}
\begin{proof}
	We first assume $ r \leq s_X $ so that $ \sigma_k > 0 $ for $ k = 1, \ldots, r $.
	
	First, the equivalence of the two expressions for $ w_j^r $ comes from $ K^*(w) \varphi_k = \sigma_k f_k $, $ \sigma_k > 0 $ for $ k = 1, \ldots, r $, and the formulas for $K^*(w) $.  Also, for $ g \in S $, \Cref{prop:POD_op_finite_dim_data} implies
	$$
	K(w^r) g = \sum_{k=1}^r \sigma_{k} (g,f_k)_{S} \varphi_k = K_r(w)g.
	$$
	Therefore, $ K(w^r) = K_r(w) $, where $ K_r(w) : S \to X $ is the $r$th order truncated SVD of the POD operator $ K(w) : S \to X $.  
	
	Next, by the Hilbert-Schmidt norm results \Cref{HS} and \Cref{prop:HS_discrete} and since the POD operator is linear in the data we have
	$$
	E_r(w^r) = \| K(w - w^r) \|^2_{\mathrm{HS}} = \| K(w) - K(w^r) \|_{\mathrm{HS}}^2 = \| K(w) - K_r(w) \|_{\mathrm{HS}}^2 = \sum_{k > r} \sigma_k^2.
	$$
	Also, since $ \| K(w) \|_{\mathrm{HS}} = \sum_{k \geq 1} \sigma_k^2 < \infty $, we have $ \sum_{k > r} \sigma_k^2 \to 0 $ as $ r $ increases.
	
	Now we show that this is the smallest value possible for the error.  Let coefficients $ \{ a_{k} \} \subset \mathbb{K} $ and basis elements $ \{ s_k \} \subset S $ and $ \{ \eta_k \} \subset X $ be given, and define the $ r $th order approximation
	$$
	  z_j^r = \sum_{k=1}^r \alpha_k s_{k,j} \eta_k  \quad  \mbox{(discrete case)},  \quad  z_j^r(t) = \sum_{k=1}^p \alpha_{k} s_{k,j}(t) \eta_k,  \quad  \mbox{(continuous case).}
	$$
	By \Cref{prop:POD_op_finite_dim_data}, $ K(z^r) $ has rank at most $ r $.  Therefore, we have
	$$
	  E_r(z^r) = \| K(w - z^r) \|^2_{\mathrm{HS}} = \| K(w) - K(z^r) \|^2_{\mathrm{HS}} \geq \sum_{k > r} \sigma_k^2.
	$$
	
	Next, if $ s_X < \infty $, then the result is true for $ r = s_X $.  Therefore, we have $ w_j = w_j^{s_X} $ for all $ j $, and this proves the result for $ r > s_X $.
\end{proof}

\bibliographystyle{plain}
\bibliography{references_POD_approx,POD_theory_new_references}

\end{document}